\documentclass[3p, sort&compress]{elsarticle}

\usepackage{lineno}
\usepackage[hyperfootnotes=false,colorlinks]{hyperref}

\journal{CMAME}

\usepackage{amsmath, amsthm, amssymb, amsfonts, verbatim, xcolor, graphicx, comment, tikz, pgfplots, mathtools, mathrsfs, ifthen, import, float, subcaption, enumitem, multirow,booktabs, bm}
\usepackage[symbol]{footmisc}
\usepackage[ruled,vlined,linesnumbered]{algorithm2e}
\SetKwComment{Comment}{/* }{ */}
\SetKwInput{KwOutput}{Output}

\usepackage{adjustbox}
\usepackage{empheq}
\usepackage{threeparttable}
\usepackage{booktabs, caption, makecell}

\pgfplotsset{compat=newest}

\usetikzlibrary{patterns}

\newtheorem{proposition}{Proposition}

\newtheorem{example}{Example}

\newcommand{\R}{\mathbb{R}}

\newcommand{\U}{\mathbb{U}}
\newcommand{\V}{\mathbb{V}}
\newcommand{\M}{\mathbb{M}}
\newcommand{\n}{\text{NN}}

\newcommand{\J}{\mathcal{J}}

\renewcommand{\L}{\mathcal{L}}
\newcommand{\F}{\mathcal{F}}

\newcommand{\D}{\mathcal{D}}

\newcommand{\review}[1]{{\color{black}{#1}}}
\usepackage[colorinlistoftodos,prependcaption, color=yellow!90!black, backgroundcolor=yellow!40!white,textcolor=black, textwidth=1.8cm,textsize=footnotesize]{todonotes}


\begin{document}

\begin{frontmatter}

\title{A Deep Double Ritz Method (D$^2$RM) for solving Partial Differential Equations using Neural Networks}

\author[1,2]{Carlos Uriarte}
\ead{carlos.uriarte@ehu.eus, curiarte@bcamath.org,  carlos.uribar@gmail.com}

\author[2,1,3]{David Pardo}
\ead{david.pardo@ehu.eus, dzubiaur@gmail.com}

\author[4,1]{Ignacio Muga}
\ead{ignacio.muga@pucv.cl, ignacio.muga@gmail.com}

\author[1,5]{Judit Muñoz-Matute}
\ead{jmunoz@bcamath.org, judit.munozmatute@utexas.edu,  judith.munozmatute@gmail.com}

\address[1]{Basque Center for Applied Mathematics (BCAM), Alameda Mazarredo 14, E48009 Bilbao, Spain}
\address[2]{Universidad del País Vasco/Euskal Herriko Unibertsitatea (UPV/EHU), Barrio Sarriena, E48940 Leioa, Spain}
\address[3]{Basque Foundation for Science (Ikerbasque), Plaza Euskadi 5, E48009 Bilbao, Spain}
\address[4]{Pontificia Universidad Católica de Valparaíso (PUCV), Avenida Brasil 2950, Valparaíso, Chile}
\address[5]{Oden Institute for Computational Engineering and Sciences (OICES), 201 E 24th St, 78712, Austin, Texas, USA}
 
\begin{abstract}
Residual minimization is a widely used technique for solving Partial Differential Equations in variational form. It minimizes the dual norm of the residual, which naturally yields a saddle-point (min-max) problem over the so-called trial and test spaces.  In the context of neural networks,  we can address this min-max approach by employing one network to seek the trial minimum,  while another network seeks the test maximizers.  However,  the resulting method is numerically unstable as we approach the trial solution. To overcome this, we reformulate the residual minimization as an equivalent minimization of a Ritz functional fed by optimal test functions computed from another Ritz functional minimization. We call the resulting scheme the Deep Double Ritz Method (D$^2$RM), which combines two neural networks for approximating trial functions and optimal test functions along a nested double Ritz minimization strategy.  Numerical results on different diffusion and convection problems support the robustness of our method, up to the approximation properties of the networks and the training capacity of the optimizers.
\end{abstract}


\end{frontmatter}

\section{Introduction}

In the last decade,  Neural Networks (NNs) have emerged as a powerful alternative for solving Partial Differential Equations (PDEs).  For example,  \cite{zhang2020meshingnet,pfaff2020learning,paszynski2021deep,sluzalec2022quasi} employ NNs to generate optimal meshes for later solving PDEs by a Finite Element Method (FEM), \cite{uriarte2022finite} proposes a Deep-FEM method that mimics mesh-refinements within the NN architecture,  \cite{omella2022r} employs a NN to generate the mesh via $r$-adaptivity,  and \cite{brevis2021machine, brevis2022neural} use NNs to improve discrete weak formulations.  Alternatively,  there exist approaches to directly represent the PDE solutions via NNs.  \review{To mention a few: \cite{raissi2019physics,qin2022rar} minimize the strong form of the residual via collocation methods;} \cite{yu2018deep} proposes a Deep Ritz Method (DRM) for symmetric and positive definite problems; \cite{taylor2022deep} introduces a Deep Fourier Method; and \cite{sirignano2018dgm, kharazmi2019variational, khodayi2020varnet,kharazmi2021hp, shang2022deep} propose (Petrov-)Galerkin frameworks in the context of trial NNs with test functions belonging to linear spaces.  \review{All these NN-based methods exhibit multiple features but also present several limitations (see, e.g.,  \cite{shin2020convergence, krishnapriyan2021characterizing,wang2022and,daw2022rethinking}). For example,  ($hp$-)VPINNs \cite{kharazmi2019variational, kharazmi2021hp} require a set of test functions which may be difficult to construct in order to guarantee optimal stability properties.}

In the context of \review{variational} residual minimization methods \cite{bramble1997least,  bochev2009least, calo2021isogeometric,calo2020adaptive,cier2021automatically},  Weak Adversarial Networks (WANs) \cite{zang2020weak, bao2020numerical} approximate both the trial and test functions via NNs.  Since the dual norm is defined as a supremum over the test space, the residual minimization reads as a min-max problem over the trial and test real Hilbert spaces,  $\U$ and $\V$, respectively.  Namely,
\begin{equation}\label{intro1}
 u^* = \arg\min_{u\in \U} \Vert Bu-l\Vert_{\V'} = \arg\min_{u\in \U} \max_{\substack{\Vert v\Vert_\V = 1}} (Bu-l)(v),
\end{equation} where $B:\U\longrightarrow\V'$ is the differential operator governing the PDE in variational form,  $u^*$ is the exact solution,  $l\in\V'$ is the right-hand side,  and $\V'$ denotes the dual space of $\V$.  This naturally leads to employing Generative Adversarial Networks (GANs) \cite{goodfellow2020generative} by approximating $u$ and $v$ with two NNs.  Unfortunately,  this approach presents a severe numerical limitation: the Lipschitz continuity constant of the test maximizers with respect to the trial functions becomes arbitrarily large when approaching the exact solution. Indeed,  the corresponding test maximizer is highly non-unique in the limit.  This provides an inherent lack of numerical stability of the method,  which is confirmed by our numerical experiments with simple model problems.

To overcome the above limitations, we reformulate the residual minimization as a minimization of a Ritz functional fed by optimal test functions \cite{demkowicz2011class,  demkowicz2014overview}. Since optimal test functions are in general unknown, we compute them for each trial function using another Ritz method. Thus, the resulting scheme is a nested double-loop Ritz minimization method: the outer loop seeks the trial solution,  while the inner loop seeks the optimal test function for each trial function.  We call this the \emph{Double Ritz Method}.

In some occasions, the trial-to-test operator that maps each trial function with the corresponding optimal test function is available. For example,  when the problem is symmetric (in which case $\U = \V$), positive definite, and we select the norm induced by the bilinear form,  the trial-to-test operator is the identity; or when selecting the strong variational formulation, the trial-to-test operator is the one given by the PDE operator. In these cases, the Double Ritz Method reduces to a single-loop Ritz minimization (see Section \ref{Generalized Ritz method} for further details).  Thus, the Double Ritz Method is a general method for solving PDEs in different variational forms,  which in some particular cases simplifies into a single-loop Ritz minimization method.

Thanks to NNs, we find a simple and advantageous computational framework to approximate and connect the trial and test functions between the outer- and inner-loop minimizations in the Double Ritz Method, a task that is difficult to tackle with traditional numerical methods.  In this work,  we propose using one network to represent the trial functions, and another network to represent the local actions of the trial-to-test operator. Thus, the composition of both networks represents the (optimal) test functions,  and we preserve the trial dependence of the test functions during the entire process.  We call the resulting NN-based method the \emph{Deep Double Ritz Method} (D$^2$RM).

While the D$^2$RM replicates existing residual minimization methods ---and related Petrov-Galerkin methods \cite{demkowicz2011class,  demkowicz2014overview, cohen2012adaptivity,  broersen2018stability, demkowicz2020dpg}--- in the context of NNs,  we fall short to provide a detailed mathematical convergence analysis. This is because our network architectures generate manifolds instead of discrete vector spaces \cite{petersen2021topological, lei2020geometric}, which represents a departure from the traditional mathematical approach. Related to this,  we encounter the usual drawbacks of lack of convexity between the trainable parameters of the network with respect to the loss functions,  which is critical to make a proper diagnosis of the optimizer during training. 

The remainder of this work is as follows.  Section \ref{section:Mathematical framework} derives the Double Ritz method from residual minimization at the continuous level and Section \ref{section:Adversarial Neural Networks} introduces it in the NN framework.  Section \ref{section:Implementation} provides implementation details and Section \ref{section:Numerical results} describes numerical experiments.  Finally,  Section \ref{section:Conclusions and future work} summarizes and concludes the work.

\section{From residual minimization to Ritz-type methods}
\label{section:Mathematical framework}

We introduce the residual minimization approach, followed by a saddle-point formulation and an alternative Double Ritz method at the continuous level.  Subsequently, we describe three particular cases for which the Double Ritz method simplifies into a single Ritz method. We follow the reasoning in \cite{demkowicz2011class, demkowicz2014overview} at the continuous level.

\subsection{Residual minimization}

We consider the following abstract variational formulation:
\begin{equation}\label{PG}
\displaystyle{\left\|
\begin{tabular}{l}
Find $u^* \in \U$ such that\\
$b(u^*,v)=l(v), \; \forall v\in \V$,
\end{tabular}
\right.}
\end{equation} where $\U$ and $\V$ are real Hilbert \emph{trial} and \emph{test} spaces, respectively, $b:\U\times \V\longrightarrow\R$ is a bilinear form, and $l:\V\longrightarrow\R$ is a continuous linear functional. Equivalently, in operator form:
\begin{equation}\label{PG_operator}
\displaystyle{\left\|
\begin{tabular}{l}
Find $u^*\in \U$ such that\\
$Bu^*=l$,
\end{tabular}
\right.} 
\end{equation}
where $B:\U\longrightarrow \V'$ is the operator defined by $(Bu)(v) := b(u,v)$, $\V'$ denotes the topological dual of $\V$, and $l\in \V'$.  The equivalent Minimum Residual formulation reads as
\begin{equation}\label{DualRes}
    u^* = \arg\min_{u\in \U} \Vert Bu-l\Vert_{\V'},
\end{equation} where $Bu-l\in\V'$ is the \emph{residual} for each trial function $u\in\U$.

To guarantee well-posedness of problem \eqref{PG_operator}, we assume
\begin{equation}
\{ v\in\V : b(u,v)=0, \forall u\in\U\}=\{0\},
\end{equation} and that $B$ is continuous and bounded from below, i.e.,
\begin{equation}
\gamma \Vert u\Vert_\U \leq \Vert Bu\Vert_{\V'} \leq M \Vert u\Vert_\U, \qquad u\in\U,
\end{equation} for some positive constants $M\geq \gamma$.  Then, the error in $\U$ is equivalent to the residual in $\V'$ in the following sense:
\begin{equation}\label{residual_error_relation}
\frac{1}{M} \Vert Bu-l \Vert_{\V'} \leq \Vert u-u^*\Vert_{\U} \leq \frac{1}{\gamma} \Vert Bu-l \Vert_{\V'}, \qquad u\in\U.
\end{equation} However,  evaluating the norm of the residual in the dual space is challenging.  In what follows, we examine two alternatives to evaluate and minimize it.

\subsection{Saddle-point problem}\label{section:Approach1}
The norm in $\V'$ is defined in terms of the norm in $\V$ as follows:
\begin{equation}\label{dual}
\Vert l\Vert_{\V'}:=\sup_{\Vert v\Vert_\V = 1} l(v),  \qquad l\in\V'.
\end{equation} Combining  \eqref{DualRes} and \eqref{dual} yields:
\begin{equation}
    u^* = \arg\min_{u\in \U} \max_{\Vert v\Vert_\V = 1} \J(u,v), \qquad \J(u,v):=(Bu-l)(v). \label{infsup_residual}
\end{equation} In the above,  the residual becomes the null operator for the exact solution, i.e., $(Bu^*-l)(v)=0$ for all $v\in \V$, which implies that the test maximizer is highly non-unique in the limit. In addition,  the operator that maps each trial function $u\in\U$ to its unitary test maximizer is not Lipschitz continuous as we approach the exact solution $u^*$ (see \ref{appendix}), leading to an unstable numerical method.

\subsection{Double Ritz Method with optimal test functions}\label{section:Double Ritz Method with optimal test functions}
The Riesz Representation Theorem allows us to work isometrically on the test space instead of on its dual. In particular, for the residual, we have
\begin{equation}
\Vert Bu-l\Vert_{\V'} = \Vert R^{-1}_\V(Bu-l)\Vert_{\V},
\end{equation} where $R_\V: \V\ni v\longmapsto (v,\cdot)_\V\in\V'$ denotes the Riesz operator.  This relation suggests considering the \emph{trial-to-test} operator $T:\U\longrightarrow\V$ defined by $T:=R^{-1}_\V B$,  since it relates the error in $\U$ with the Riesz representative of the residual in $\V$, i.e.,
\begin{equation} \label{error_residual_representation}
T(u-u^*) = R^{-1}_\V (Bu-l) \in \V.
\end{equation} The images of trial functions through $T$ are known as \emph{optimal test functions}, and they allow us to rewrite \eqref{PG_operator} in terms of the symmetric and positive definite variational problem (cf. \cite{demkowicz2014overview})
\begin{equation}\label{Tproblem}
\displaystyle{\left\|
\begin{tabular}{l}
Find $u^* \in \U$ such that\\
$(T u^*, T u)_\V = l(T u), \; \forall u\in \U$,
\end{tabular}
\right.} 
\end{equation} which is equivalent to minimize the following quadratic functional:
\begin{equation} \label{Ritz_u}
    u^* = \arg\min_{u\in \U} \F_T(u), \qquad \F_T(u):=(\F\circ T)(u)=\frac{1}{2}\Vert T u\Vert_\V^2 -l(T u).
\end{equation} Here, $\F:\V\longrightarrow\R$ is the traditional Ritz functional given by $\F(\cdot)=\frac{1}{2}\Vert \cdot\Vert_\V^2-l(\cdot)$.  Indeed, $\F_T$ acts as a generalization of the Ritz functional into more general problems (cf. Section \ref{Generalized Ritz method}).  The relation between residual minimization \eqref{DualRes} and the generalized Ritz minimization \eqref{Ritz_u} is the following:
\begin{align}\label{residual_min_Ritzu}
\F_T(u)-\F_T(u^*) = \frac{1}{2}\Vert Bu-l\Vert_{\V'}^2,
\end{align} which easily deduces from the definition of $\F_T$.

Now,  the challenge is to compute $Tu\in\V$ when iterating along $u\in \U$,  since $T$ is unavailable in general.  Notice that finding $Tu\in\V$ is equivalent to solving the symmetric and positive definite variational problem
\begin{equation}\label{VarVh}
\displaystyle{\left\|
\begin{tabular}{l}
Find $Tu \in \V$ such that\\
$(T u,  v)_\V = b(u,v), \; \forall v\in \V$,
\end{tabular}
\right.} 
\end{equation} which, as before,  is equivalent to the following Ritz minimization:
\begin{equation} \label{Ritz_v} 
    \text{Given } u\in\U:\quad T u = \arg\min_{v\in \V} \L_u(v), \qquad \L_u(v):=\frac{1}{2}\Vert v\Vert_\V^2 - b(u,v).
\end{equation} However,  characterizing $Tu$ by means of \eqref{Ritz_v} remains impractical when iterating along $u\in\U$ to minimize \eqref{Ritz_u}, since directional derivatives of $T$ at $u\in\U$ are inaccessible.  To overcome this,  instead of iterating along $\V$ to seek the optimal test function for a given $u\in\U$, we consider a suitable set $\M$ of operators from $\U$ to $\V$ whose directional derivatives are accessible and the attainability of $Tu$ is guaranteed for all $u\in\U$. Then, we reformulate \eqref{Ritz_v} as seeking a candidate $\tau_u\in\M$ that when acting on $u\in\U$ returns $T u\in\V$. Namely,
\begin{equation} \label{Ritz_T}
    \text{Given } u\in\U: \quad \tau_{u} = \arg\min_{\tau\in\M} \L_u(\tau(u)) \quad \text{ and } \quad \tau_u(u)=Tu\in\V.
\end{equation} Notice that by performing \eqref{Ritz_T},  we ensure that the minimizer $\tau_u$ acts as $T$ \emph{only} at the given $u\in\U$,  i.e., $\tau_u(w)$ may be far from $T w$ whenever $u\neq w$.  Moreover,  depending on the construction of $\M$,  $\tau_u$ is possibly non-unique.

Hence,  by performing a nested minimization of \eqref{Ritz_T} within \eqref{Ritz_u},  we solve the problem at hand without explicitly dealing with the full operator $T$.  We call this the \emph{Double Ritz Method}. Algorithm \ref{alg:outer-inner loop} depicts its nested-loop optimization strategy that iterates separately either with elements in $\U$ or in $\M$.

\begin{algorithm}
\caption{Double Ritz Method}\label{alg:outer-inner loop}
Initialize $u\in \U$ and $\tau\in\M$\;
\While{\upshape not converged}{
    Find $\tau_{u} \in \arg\min_{\tau\in \M} \L_u(\tau u)$\;
    $u \leftarrow$ following candidate in $\U$ minimizing $\F_{\tau_u}(u)$\;
}
\Return $u$
\end{algorithm}

\subsection{Generalized Ritz Methods}\label{Generalized Ritz method}

We show three scenarios where the Double Ritz Method simplifies into a single Ritz minimization, e.g., when the trial-to-test operator is available in \eqref{Ritz_u}.

\begin{enumerate}
\item[(a)] \textbf{(Traditional) Ritz Method.} If the bilinear form $b(\cdot,\cdot)$ is symmetric and positive definite, we have $\V=\U$. Selecting the inner product as the bilinear form yields $T u = u$ for all $u\in \U$,  i.e., $T$ is the identity operator.

\item[(b)] \textbf{Strong formulation.} Let $A:\D(A)\longrightarrow L^2(\Omega)$ denote a PDE operator with domain $\D(A)$.  If we select $b(u,v):=(Au,v)_{L^2(\Omega)}$, $\U:=\D(A)$ endowed with the graph norm,  and $\V:=L^2(\Omega)$, we obtain $Tu = Au$ for all $u\in \U$, i.e., $T$ is the PDE operator.

\item[(c)] \textbf{Ultraweak formulation.} Let $A':\D(A')\longrightarrow L^2(\Omega)$ denote the adjoint operator of the PDE operator $A:\D(A)\longrightarrow L^2(\Omega)$, i.e., $(Au,v)_{L^2(\Omega)}=(u,A' v)_{L^2(\Omega)}$. If we select $b(u,v):=(u,A' v)_{L^2(\Omega)}$, $\U:=L^2(\Omega)$,  and $\V:=\D(A')$ endowed with the graph norm, we obtain $A' Tu = u$ for all $u\in \U$. 

Following the ideas from \cite{brunken2019parametrized, demkowicz2020dpg}, we can first solve for the optimal test function of the trial solution:
\begin{equation}\label{adjoint_minimization}
Tu^*=\arg\min_{v\in \V} \F'(v), \qquad \F'(v):=\frac{1}{2}\Vert A' v\Vert^2_{L^2(\Omega)}-l(v), \end{equation}
and then apply $A'$ to the minimizer to recover the trial solution: $u^*=A'Tu^*$.  We call \emph{Adjoint Ritz Method} to this Ritz minimization with post-processing based on the use of the $A'$ operator.
\end{enumerate}

To illustrate the above three cases, we show a simple PDE problem with different variational formulations. 

\begin{example}[Poisson's Equation]\label{Poisson's Equation}
Let $f\in L^2(0,1)$ and consider the following 1D pure diffusion equation with homogeneous Dirichlet boundary conditions over $\Omega=(0,1)$:
\begin{equation}\label{1DLap}
\begin{cases}
-u''=f, &\\
u(0)=u(1)=0. &
\end{cases}
\end{equation} We multiply the PDE by a test function and integrate over $\Omega$. Depending on the number of times we integrate by parts to derive the variational formulation, we obtain the following:
\begin{enumerate}
\item \textbf{Strong formulation.} Without integration by parts. Then, $\U:=H^2(0,1)\cap H^1_0(0,1)$, $\V:=L^2(0,1)$, $b(u,v):=\int_0^1 -u''v$, and $Tu=-u''$ for all $u\in \U$. This is the case {\upshape (b)} above.
\item \textbf{Weak formulation.} We integrate by parts once, passing one derivative from the trial to the test function. Then, $\U:=H_0^1(0,1)=:\V$, $b(u,v):=\int_0^1 u'v'$, $(v_1, v_2)_\V:=b(v_1,v_2)$, and $Tu=u$ for all $u\in \U$.  This is the case {\upshape (a)} above.
\item \textbf{Ultraweak formulation.}  We integrate by parts twice, passing the two derivatives of the trial to the test function. Then, $\U:=L^2(0,1)$, $\V:=H^2(0,1)\cap H^1_0(0,1)$,  $b(u,v):=\int_0^1-uv''$,  $(v_1,v_2)_\V:=(-v_1'',-v_2'')_{L^2(\Omega)}$,  and $-(Tu)''=u$ for all $u\in \U$. This is the case {\upshape (c)} above.
\end{enumerate}
\end{example}

All three variational formulations are valid, although they exhibit different convergence behaviors that go beyond the scope of this work.

\section{Approximation with Neural Networks}
\label{section:Adversarial Neural Networks}

In practice,  instead of seeking along $\U$, $\V$,  and $\M$, we seek along corresponding computationally accessible subsets.  This section elaborates on employing NNs for the computability of the proposed methods.

\subsection{Discretization}

Let $u_\n:\Omega\longrightarrow\R$ denote a NN with corresponding set of learnable parameters $\theta_u$, where $\Omega$ denotes the spatial domain related to a boundary value problem.  We use this NN to approximate trial functions. Thus,
\begin{equation}
    \U_{\n} := \{u_\n(\cdot\;;\theta_u):\Omega\longrightarrow\R \}_{\theta_u\in \Theta_u}
\end{equation} is the set of all possible configurations of learnable parameters (a.k.a \emph{realizations}) for $u_{\n}$, where $\Theta_u$ denotes the domain of $\theta_u$.  Similarly, $v_{\n}:\Omega\longrightarrow\R$ and $\tau_{\n}:\U_{\n}\longrightarrow \V_{\n}$ will denote NNs with corresponding sets of learnable parameters $\theta_v\in\Theta_v$ and $\theta_T\in\Theta_T$, and sets of realizations $\V_\n$ and $\M_\n$, respectively. In the end, the architectures of the NNs determine the approximation capacity of our methods, but its study is beyond the scope of this work (see, e.g.,  \cite{hornik1989multilayer,  csaji2001approximation,  kratsios2020non}). 

We select fully-connected architectures for the NNs, namely,
\begin{subequations}
\begin{alignat}{4} \label{FC_architecture}
    & u_{\n}(x_1) &\ =\ & L_{k+1} \circ L_k\circ \cdots \circ L_2 \circ L_1(x_1), &\qquad k\geq 1, \\
   & \phantom{u_\n(}x_i\phantom{)} & \ = \ & L_i(x_{i-1}) = \varphi(W_i x_{i-1} + b_i), &\qquad 2\leq i\leq k,\\
    & & & L_{k+1}(x_{k}) = W_{k+1} x_{k}, &\qquad
\end{alignat} 
\end{subequations} where $x_1$ is the input vector of the network, $W_i x_{i-1}+b_{i}$ is an affine transformation, $\varphi$ is an activation function,  $W_{k+1} x_k$ is a linear combination of the components of vector $x_k$ with coefficients in $W_{k+1}$, and $k$ is the depth of the NN. In particular, the set of learnable parameters consists of 
\begin{equation}
\Theta_u:=\{W_i\in\R^{n_i\times n_{i-1}}, b_i\in\R^{n_i\times 1} : 1\leq i\leq k\} \cup \{W_{k+1}\in\R^{1\times n_{k}}\}, 
\end{equation} where $n_{i-1}$ and $n_i$ are the input and output dimensions of layer $L_i$. 

To ensure the containment of the set of realizations within the trial space (i.e., $\U_\n\subset \U$), we select convenient activation functions (e.g., $\varphi = \tanh$) and strongly impose Dirichlet boundary conditions whenever is needed. For the latter, we redefine the above free-boundary NN architecture as 
\begin{equation}
    u_\n(x) \leftarrow u_D(x) + D(x) \cdot u_\n(x),
\end{equation} where $D(x)$ is an auxiliary continuous function such that $D\vert_{\Gamma_D} = 0$ and $D\vert_{\overline{\Omega}\setminus\Gamma_D} \neq 0$, with $\Gamma_D$ denoting the Dirichlet boundary, and $u_D$ is a lift of $u\vert_{\Gamma_D}$.  We consider similar architecture settings for $v_\n$ (see Figure \ref{figure:implementation_sketch_MIN_MAX}) and $\tau_\n$ (see Figure \ref{figure:implementation_sketch_MIN_MIN}).

Table \ref{table_methods_NNs} shows the methods introduced at Sections \ref{section:Approach1}-\ref{Generalized Ritz method} in the framework of NNs.

\begin{table}[htb]
\centering
\resizebox{\textwidth}{!}{%
\begin{tabular}{@{}|r|l|l|@{}}
\toprule
\multicolumn{1}{|c|}{\textbf{Method}} & \multicolumn{1}{c|}{\textbf{Functional(s)}}                                                                                                                                                                                                                    & \multicolumn{1}{c|}{\textbf{Optimization}}                                                                                                                                                                                  \\ \midrule
Weak Adversarial Networks             & $\displaystyle \J(u_\n,v_\n)=(Bu_\n-l)\left(\frac{v_\n}{\lVert v_\n\rVert_V}\right)$                                                                                                                                                                     & $\displaystyle u^*_\n = \arg\min_{u_\n\in\U_\n} \max_{v_\n\in\V_\n} \J(u_\n,v_\n)$                                                                                                                         \\ \midrule
Generalized Deep Ritz Method          & \begin{tabular}[c]{@{}l@{}}$\displaystyle \;\;\;\;\;\;\F_T(u_\n)=\tfrac{1}{2} \lVert Tu_\n\rVert_V^2 - l(Tu_\n)$                                                                                                                                                                                    \end{tabular} & $\displaystyle u^*_\n = \arg\min_{u_\n\in\U_\n} \F_T(u_\n)$                                                                                                                                                                 \\ \midrule
Deep Double Ritz Method               & \begin{tabular}[c]{@{}l@{}}$\displaystyle \;\;\;\F_{\tau_\n}(u_\n)=\tfrac{1}{2} \lVert \tau_\n(u_\n)\rVert_V^2 - l(\tau_\n(u_\n))$\vspace{0.2cm}  \\ $\displaystyle \;\;\;\L_{u_\n}(\tau_\n) =\tfrac{1}{2} \lVert \tau_\n(u_\n)\rVert_V^2 - (Bu_\n)(\tau_\n(u_\n))$\end{tabular} & \begin{tabular}[c]{@{}l@{}}$\displaystyle u^*_\n = \arg\min_{u_\n\in\U_\n} \F_{\tau_{u_\n}}(u_\n)$ \\ $\displaystyle \tau_{u_\n} = \arg\min_{\tau_\n\in\M_\n} \L_{u_\n}(\tau_\n(u_\n))$\end{tabular} \\ \bottomrule
\end{tabular}%
}
\caption{The min-max approach,  the Generalized Ritz Method, and the Double Ritz Method in the context of NNs: \emph{Weak Adversarial Networks (WANs)}, \emph{the Generalized Deep Ritz Method (GDRM)}, and \emph{the Deep Double Ritz Method (D$^2$RM)}. }
\label{table_methods_NNs}
\end{table}

\subsection{Differentiation and integration}\label{section:Quadrature rules}

Forms $b(\cdot,\cdot)$, $l(\cdot)$, and $\Vert \cdot\Vert_\V$ consist of definite integrals of combinations of NNs and/or their derivatives.  We perform \emph{automatic-differentiation} to evaluate the derivatives of the NNs \cite{baydin2018automatic, margossian2019review},  and approximate corresponding integrals via a quadrature rule:
\begin{equation}
    \int_\Omega I(x) dx \approx \sum_{i=1}^N \omega_i \cdot I(x_i).
\end{equation} Here,  $I$ is the integrand and $\omega_i$ is the integration weight related to the integration node $x_i\in\Omega$.  Thus,  the quality of the approximation depends on the sample and the weight selection.

In Monte Carlo integration, the weights act as uniform averages of the sum of evaluations of the integrand ($\omega_i = \text{Vol}(\Omega)/N$) \cite{davis2007methods, leobacher2014introduction}.  It allows to select a new random sample of nodes for each integration during the training and, therefore,  circumvents the well-known overfitting problem \cite{rivera2022quadrature}. Unfortunately,  Monte Carlo integration requires an immense sample size to ensure admissible integration errors,  which impacts the computational memory resources and execution speed during training.  In addition,  Monte Carlo integration has difficulties when dealing with singular solutions that we consider in this work.

To maintain the advantages of Monte Carlo integration, reduce the runtime and sample size (from thousands to hundreds of points), and control the integration error around singularities, we consider a \emph{composite intermediate-point quadrature rule} (see Algorithm \ref{alg:intermediate-point quadrature rule} for the domain $\Omega=(0,1)$) that generates random integration points following a beta $\beta(a,b)$ probability distribution \cite{gupta2004handbook, weinzierl2000introduction}. We tune the hyperparameters $a$ and $b$ conveniently according to our problem specifications.

\begin{algorithm}
\caption{Randomized composite intermediate-point quadrature rule to approximate $\int_0^1 I(x) dx$}\label{alg:intermediate-point quadrature rule}
Generate $x_i\in (0,1)$ for $1\leq i\leq N$\;
Sort $\{x_i:1\leq i\leq N\}$ so that $x_{i-1}<x_i$ for all $1\leq i\leq N$\;
Evaluate $I$ in $\{x_i:1\leq i\leq N\}$\;
Define $m_0:=0$,  $m_i:=(x_{i-1}+x_i)/2$ for $1\leq i\leq N-1$, and $m_{N}:=1$\;
Define $\omega_i:=m_i-m_{i-1}$ for $1\leq i\leq N$\;
\Return $\sum_{i=1}^N \omega_i\cdot I(x_i)$\;
\end{algorithm}

\review{
For $\Omega=(0,1)\times(0,1)$,  we perform lines 1 and 2 of Algorithm \ref{alg:intermediate-point quadrature rule} on each axis,  and adapt lines 3-6 to a cartesian-product structure.}

\subsection{Loss functions and gradient-descent/ascent}
\label{section:Loss functions and gradient-descent/ascent optimizations}

We define the loss functions as sampling-dependent functionals that represent the quadrature rules approximating the analytic integrals (recall Table \ref{table_methods_NNs}):
\begin{subequations} \label{equation:integral_functional vs losses}
\begin{alignat}{4}
\J(u_\n, v_\n)&\approx \widehat{\J}(\{x_i\}_i;\theta_u,\theta_v), \label{equation:integral_functional vs losses a}\\
\F_T(u_\n)&\approx \widehat{\F}_T(\{x_i\}_i;\theta_u),\\
\F_{\tau_\n}(u_\n)&\approx \widehat{\F}_{\tau_\n}(\{x_i\}_i;\theta_u),\\
\L_{u_\n}(\tau_\n(u_\n))&\approx \widehat{\L}_{u_\n}(\{x_i\}_i;\theta_\tau). \label{equation:integral_functional vs losses d}
\end{alignat} 
\end{subequations} The right-hand sides of equations \eqref{equation:integral_functional vs losses a}-\eqref{equation:integral_functional vs losses d} represent the loss functions, while the left-hand sides represent the corresponding analytic integrals. 

Now, we detail three different optimization strategies for training the NNs:
\begin{itemize}
\item \textbf{Weak Adversarial Networks (WANs).} We use the Stochastic Gradient-Descent/Ascent (SGD/A) optimizer \cite{ruder2016overview} to readjust the learnable parameters as follows:
\begin{subequations}
\begin{align}
    \theta_u \leftarrow \theta_u - \eta_u \frac{\partial\widehat{\J}}{\partial\theta_u}(\{x_i\};\theta_u, \theta_v), \label{gradient-descent}\\
    \theta_v \leftarrow \theta_v + \eta_v \frac{\partial\widehat{\J}}{\partial\theta_v}(\{x_i\};\theta_u, \theta_v),
    \label{gradient-ascent}
\end{align}
\end{subequations} where $\eta_u,\eta_v>0$ denote learning-rate coefficients.  We perform multiple iterations on the ascent for each iteration on the descent (see Algorithm \ref{alg:nested_minmax}). We call \emph{outer} and \emph{inner loops} to the minimization and maximization processes,  respectively,  because of the nested optimization structure.

\begin{algorithm}[htb]
\caption{WANs training}\label{alg:nested_minmax}
Initialize $\theta_u\in\Theta_u$ and $\theta_v\in\Theta_v$\;
\Comment{Outer loop}
\While{\upshape not converged}{
Select random sampling $\{x_i\}_i\subset\Omega$\;
$\theta_u \leftarrow \theta_u - \eta_u \frac{\partial\widehat{\J}}{\partial\theta_u}(\{x_i\}_i; \theta_u, \theta_v)$\;
 \Comment{Inner loop}
\While{\upshape not converged}{
Select random sampling $\{x_i\}_i\subset\Omega$\;
$\theta_v \leftarrow \theta_v + \eta_v \frac{\partial\widehat{\J}}{\partial\theta_v}(\{x_i\}_i;\theta_u, \theta_v)$\;
  }
}
\Return $\theta_u$
\end{algorithm}

\pagebreak

\item \textbf{Deep Double Ritz Method (D$^2$RM).} The training is similar as in Algorithm \ref{alg:nested_minmax}, but modifying the loss when jumping from the outer to the inner loop, and performing only gradient-descents for minimizations (see Algorithm \ref{alg:nested_minmin}). 

\begin{algorithm}[htb]
\caption{D$^2$RM training}\label{alg:nested_minmin}
Initialize $\theta_u\in\Theta_u$ and $\theta_\tau\in\Theta_\tau$\;
\Comment{Outer loop}
\While{\upshape not converged}{
Select random sampling $\{x_i\}_i\subset\Omega$\;
$\theta_u \leftarrow \theta_u - \eta_u \frac{\partial\widehat{\F}_{\tau_\n}}{\partial\theta_u}(\{x_i\}_i; \theta_u,\theta_\tau)$\;
\Comment{Inner loop}
\While{\upshape not converged}{
Select random sampling $\{x_i\}_i\subset\Omega$\;
$\theta_\tau \leftarrow \theta_\tau - \eta_\tau \frac{\partial\widehat{\L}_{u_\n}}{\partial\theta_\tau}(\{x_i\}_i; \theta_u, \theta_\tau)$\;
  }
}
\Return $\theta_u$
\end{algorithm}

\item \textbf{Generalized Deep Ritz Method (GDRM).} The optimization consists of a single-loop minimization (see Algorithm \ref{alg:single_min}) since the trial-to-test operator $T$ is explicitly available.  In particular, when the bilinear form is symmetric and positive definite  ---recall Section \ref{Generalized Ritz method}, item (a)---,  this method is the well-known Deep Ritz Method (DRM) \cite{yu2018deep}.

\begin{algorithm}[htb]
\caption{GDRM training}\label{alg:single_min}
Initialize $\theta_u\in\Theta_u$\;
\While{\upshape not converged}{
Select random sampling $\{x_i\}_i\subset\Omega$\;
$\theta_u \leftarrow \theta_u - \eta_u \frac{\partial\widehat{\F}_T}{\partial\theta_u}(\{x_i\}_i; \theta_u)$\;
}
\Return $\theta_u$
\end{algorithm}

\item \textbf{Adjoint Deep Ritz Method (DRM$)'$.} This method consists of a single-loop minimization (as in Algorithm \ref{alg:single_min}) with post-processing, only valid in ultraweak formulations.  Recall Section \ref{Generalized Ritz method}, item (c).

\end{itemize}

\subsection{Optimization}\label{section:optimization}

We select the Adam optimizer \cite{kingma2014adam} to carry out our experiments. In the case of nested optimizations, we select fully independent Adam optimizers for the outer and inner loops. We establish an accumulated maximum number of iterations for both nested loops, and we fix four inner-loop iterations for each outer-loop iteration (as considered in \cite{zang2020weak,  bao2020numerical} for WANs).  In the D$^2$RM,  we assume that a slight modification in $\theta_u$ translates into a slight variation in $u_\n$.  Then,  the variation in $Tu_\n$ is small because $T$ is continuous and, therefore, we expect that $\tau_\n(u_\n)$ can provide a good approximation of $Tu_{\n}$ after a small number of iterations in $\theta_\tau$.  

\section{Implementation}
\label{section:Implementation}

We use the Tensorflow 2 library (TF2) \cite{abadi2016tensorflow} within Python to implement our neural network architectures and loss functions, and to manage the random creation and flow of data. Specifically, we accommodate all of our implementations to Keras (\texttt{tf.keras}), a sublibrary of TF2 specially designed for executions in \emph{graph mode} that exploits high-performance computing \cite{joseph2021keras} (e.g.,  parallelized computing within GPUs \cite{meng2017training,  sergeev2018horovod}).

\subsection{Samples generation, input batch flow, encapsulation of models, and optimization}

Our inputs to networks are samples over the domain $\Omega$.  After feeding our networks with the batch of inputs, we combine the batch of outputs in a single loss prediction. We encapsulate the networks and losses inside a main model whose input and output are the batch of samples and the loss prediction, respectively. From the loss prediction, we optimize the learnable parameters of the networks (i.e, we fit the learnable parameters to the data) with the Keras built-in Adam optimizer (\texttt{tf.keras.optimizers.Adam}).  Figure \ref{figure:implementation_sketch_optimization} illustrates the described process at each training iteration.  

\begin{figure}[htb]
\centering
\scalebox{0.64}{\includegraphics{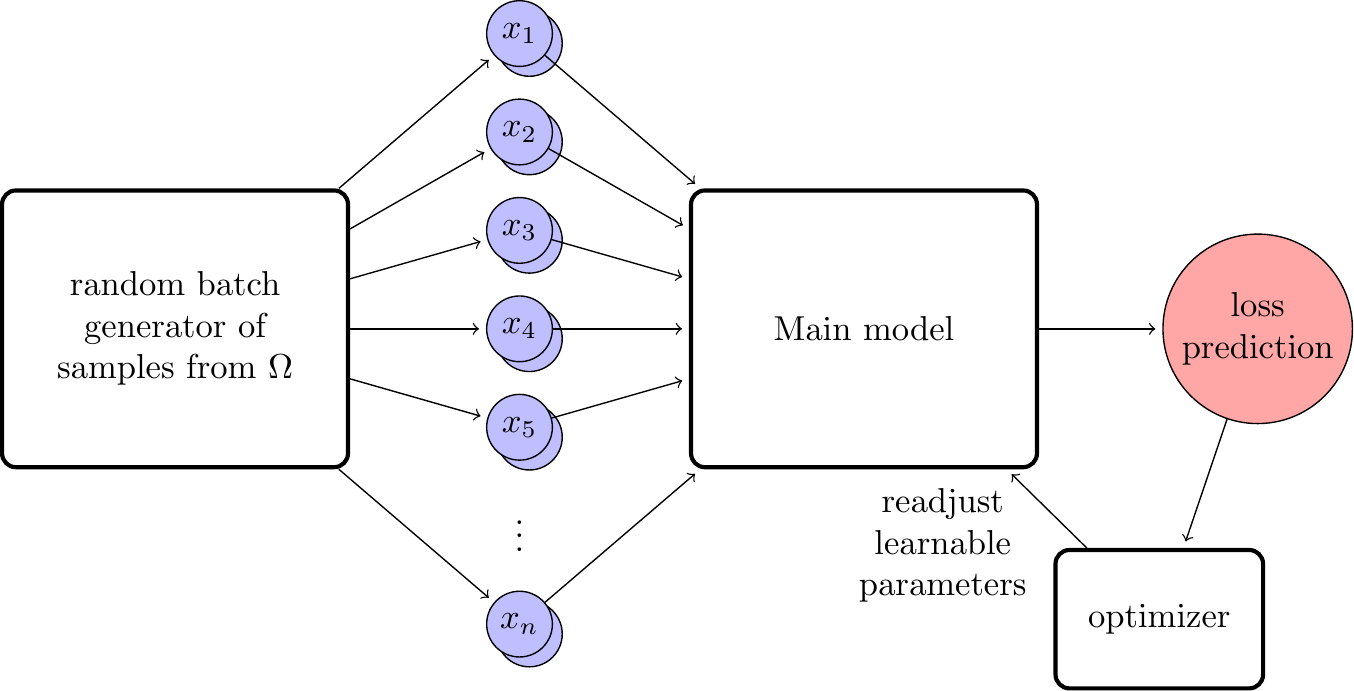}}
\caption{General flux sketch for the proposed methods at each training iteration.}
\label{figure:implementation_sketch_optimization}
\end{figure}

For the maximization involved in WANs, we reverse the sign of the gradients so that when feeding the (default) gradient-descent-based optimizer, it instead performs a gradient-ascent. 

\subsection{Design of models by methods}

We implement networks, losses, and operators as models and layers in TF2 from the redefinitions of the corresponding Keras base classes (\texttt{tf.keras.Model} and \texttt{tf.keras.Layer}). In WANs, we implement $u_\n$ and $v_\n$ as two independent models that are combined via a common non-trainable layer for the loss (see Figure \ref{figure:implementation_sketch_MIN_MAX}). In the DRM, we implement $u_\n$ as a model that subsequently connects with two non-trainable layers for the trial-to-test operator and the loss, respectively (see Figure \ref{figure:implementation_sketch_MIN}).  In the D$^2$RM, we implement $u_\n$ and $\tau_\n$ as two sequential models whose output feeds into two separate losses (see Figure \ref{figure:implementation_sketch_MIN_MIN}).  The loss functions are implemented as latent outputs of the main model, which together with a TF2-suitable boolean variable, activate and deactivate alternatively. Despite the significant compilation time that the two-branch model (for the D$^2$RM) takes compared with one-branch models (for WANs and the DRM)\footnote{Around one minute for the D$^2$RM vs. a couple of seconds for WANs and the DRM.},  its fitting execution in graph mode is as fast as that of one-branch models.

\begin{figure}[p]
\centering
\scalebox{0.7}{\includegraphics{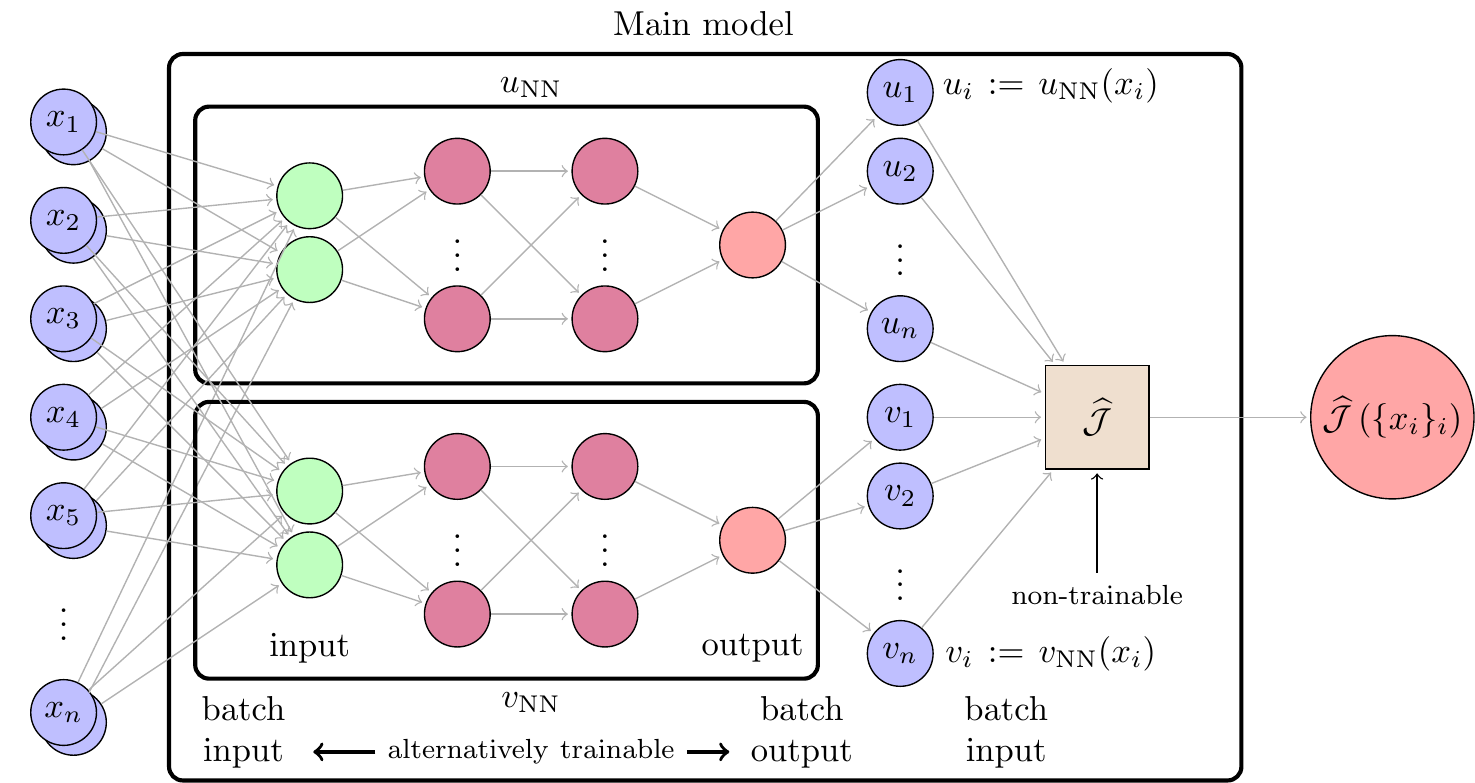}}%
\caption{Main model architecture for WANs. It consists of two independent NNs, $u_\n$ and $v_\n$,  combined via the loss $\widehat{\J}$.}
\label{figure:implementation_sketch_MIN_MAX}
\end{figure}
\begin{figure}[p]
\centering
\scalebox{0.7}{\includegraphics{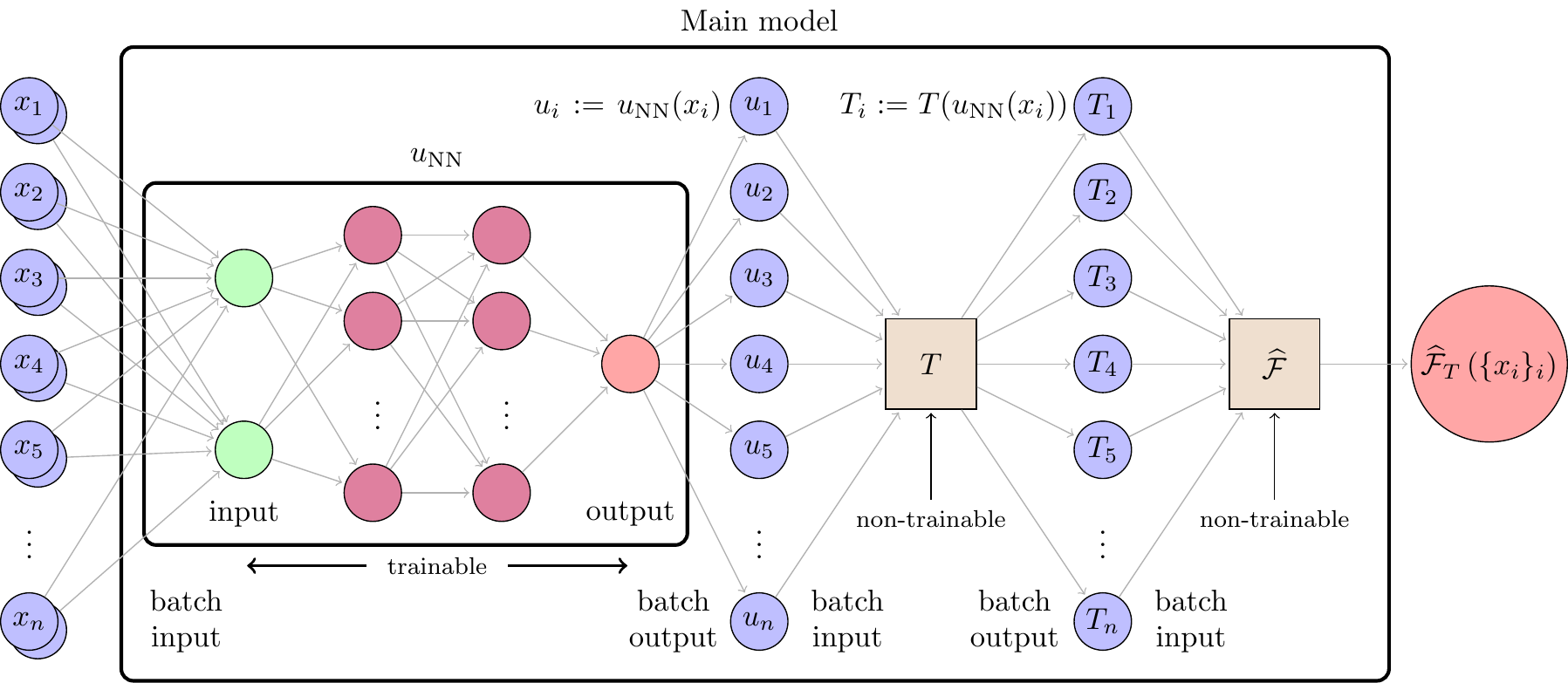}}%
\caption{Main model architecture for the DRM. It consists of a NN, $u_\n$, composed with $T$ and the loss $\widehat{\F}$.}
\label{figure:implementation_sketch_MIN}
\end{figure}
\begin{figure}[p]
\centering
\resizebox{\textwidth}{!}{%
\includegraphics{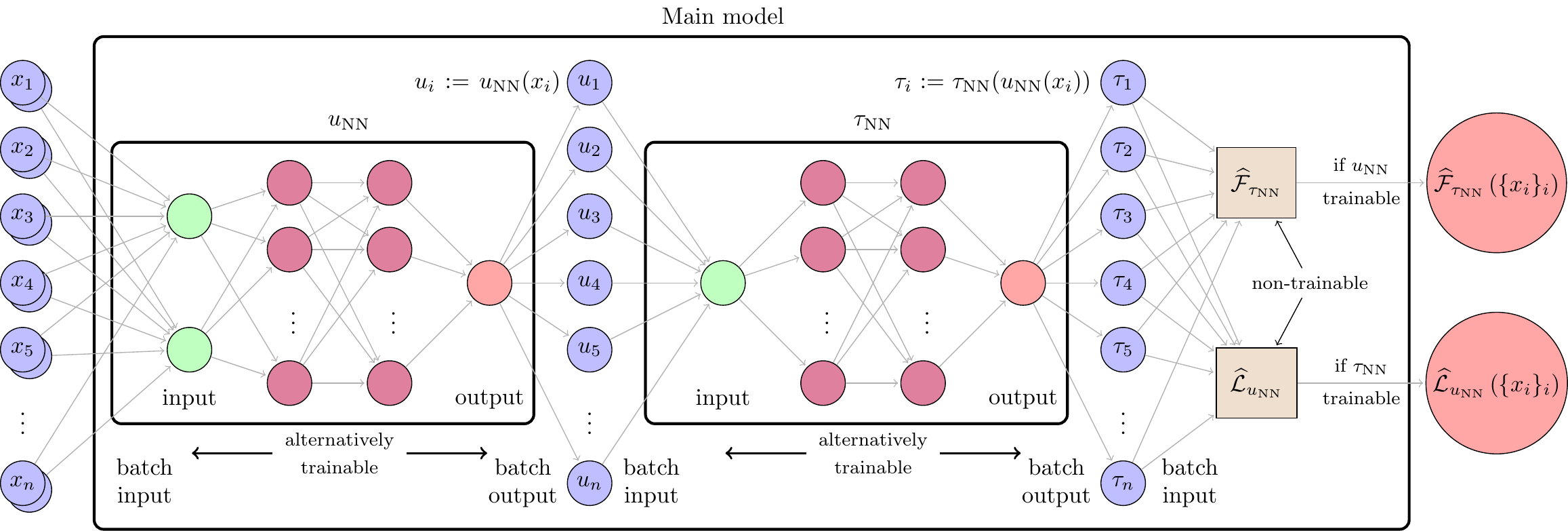}%
}
\caption{Main model architecture for the D$^2$RM. It consists of two NNs, $u_\n$ and $\tau_\n$,  equipped with losses $\widehat{\F}_{\tau_\n}$ and $\widehat{\L}_{u_\n}$.}
\label{figure:implementation_sketch_MIN_MIN}
\end{figure}

\subsection{Graph-mode execution dynamics}

We employ \emph{callbacks} to avoid interrupting the graph execution mode carried out by the Keras fitting instruction (\texttt{.fit}). Callbacks act during fitting and enable accessing certain elements of (main) models and modifying them. We utilize callbacks for loss monitoring (for WANs,  the DRM,  and the D$^2$RM), to activate or deactivate the trainability of the networks and switch between optimizers (for WANs and the D$^2$RM), or to interchange losses (for the D$^2$RM) during training when iterating either over the outer or the inner loop.

\section{Numerical results}
\label{section:Numerical results}

We show numerical experiments to compare the methods introduced above.  Section \ref{section51} considers a simple model problem and initally compares WANs, the DRM,  and the D$^2$RM.  Section \ref{section:Comparison of performances by approaches} makes a more profound comparison on a parametric problem.  Section \ref{section:Diffusion equation with a Dirac delta source} and \ref{section:Advection equation in ultraweak formulation} consider sources that lead to singular problems in pure diffusion and convection equations, respectively.  Finally, Section \ref{2D section} shows results on a 2D pure convection problem. 

\subsection{Initial comparison of WANs,  the DRM, and the D$^2$RM on a simple problem}\label{section51}

We select model problem \eqref{1DLap} in weak form with source $f=-2$, so the analytic solution is $u^*=x(x-1)$. Hence, $\U=H^1_0(0,1)=\V$ and $T$ is the identity operator.  We solve this problem employing WANs,  the DRM, and the D$^2$RM.

We select a two-layer fully-connected NN with 20 neurons on each layer and $\tanh$ activation functions for the architectures of $u_\n$, $v_\n$,  and $\tau_\n$.  We perform $200$ iterations for $u_\n$ in the three methods: WANs,  the DRM, and the D$^2$RM.  Since in WANs and the D$^2$RM we established four iterations to approximate the test maximizers (recall Section \ref{section:optimization}), we end up with a total of $1\mathord{,}000$ training iterations ($200$ for the trial function, and $800$ for the test functions). Moreover, we select batches of size $200$ for the training, and a uniform distribution for the sample generation. 

Figure \ref{figure:WANs_DRM_DDRM_x(x-1)} shows the $u_\n$ network predictions and errors of the three methods at the end of the training.  We observe that WANs produce a larger error than the DRM and the D$^2$RM.

\begin{figure}[htbp]
\centering
\begin{subfigure}{\textwidth}
\centering
\includegraphics{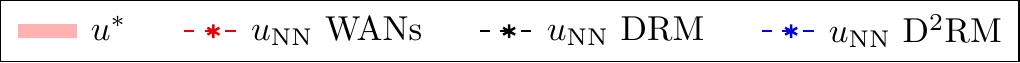}
\end{subfigure}\vskip 1em%
\begin{subfigure}[t]{0.3\textwidth}
\centering
\includegraphics{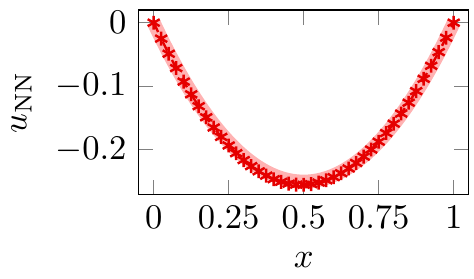}
\end{subfigure}\hskip 1em%
\begin{subfigure}[t]{0.3\textwidth}
\centering
\includegraphics{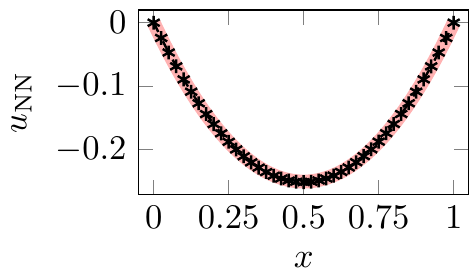}
\end{subfigure}\hskip 1em%
\begin{subfigure}[t]{0.3\textwidth}
\centering
\includegraphics{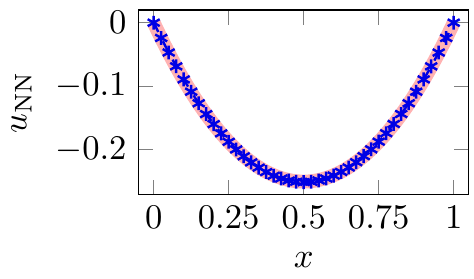}
\end{subfigure}\vskip 1em%
\begin{subfigure}[t]{0.3\textwidth}
\centering
\includegraphics{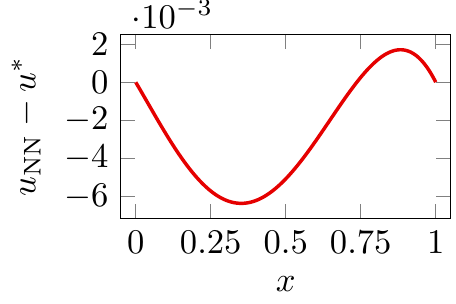}
\end{subfigure}\hskip 1em%
\begin{subfigure}[t]{0.3\textwidth}
\centering
\includegraphics{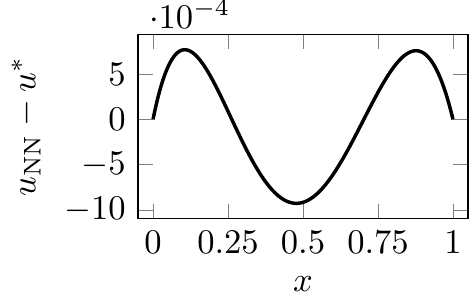}
\end{subfigure}\hskip 1em%
\begin{subfigure}[t]{0.3\textwidth}
\centering
\includegraphics{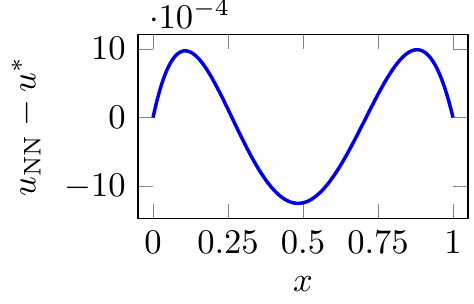}
\end{subfigure}\vskip 1em%
\caption{Trial network predictions and errors in WANs, the DRM, and the D$^2$RM at the end of the training in problem \eqref{1DLap} with analytic solution $u^*=x(x-1)$.}
\label{figure:WANs_DRM_DDRM_x(x-1)}
\end{figure}

Figure \ref{figure:WAN_x(x-1)_loss} shows the loss evolution for WANs.  Every five iterations, the loss decreases (iterations in $u_\n$), and increases in the remaining ones (while freezing the approximate solution and looking for the test maximizer).  From iteration $500$ onwards, the loss stops improving and oscillates above the optimal value. 

\begin{figure}[htbp]
\centering
\begin{subfigure}[b]{\textwidth}
\centering
\includegraphics{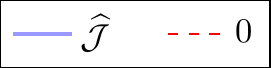}
\end{subfigure}\vskip 1em%
\begin{subfigure}[b]{0.8\textwidth}
\centering
\includegraphics{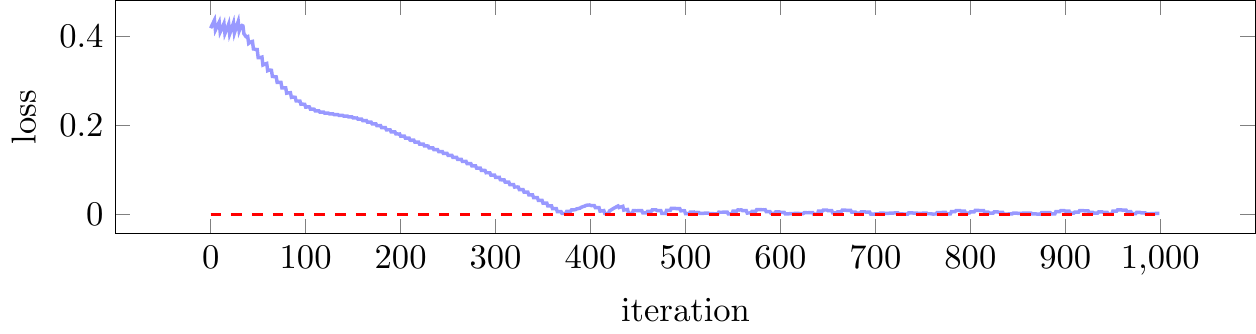}
\end{subfigure}\vskip 1em
\begin{subfigure}[b]{0.4\textwidth}
\centering
\includegraphics{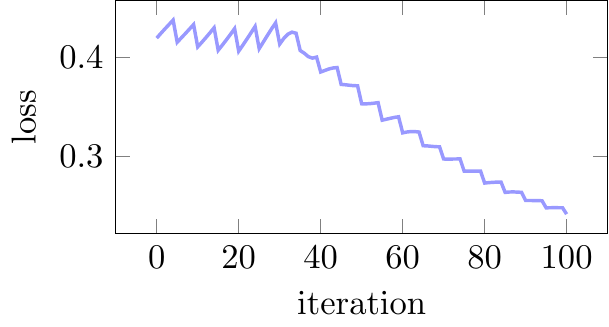}
\end{subfigure}\hskip 2em
\begin{subfigure}[b]{0.4\textwidth}
\centering
\includegraphics{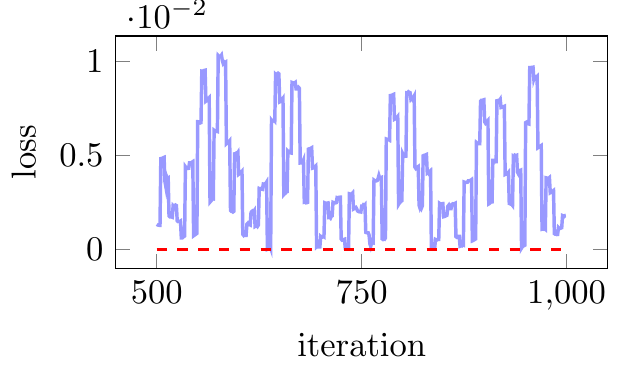}
\end{subfigure}
\caption{Loss evolution during the WANs training for problem \eqref{1DLap} with analytic solution $u^*=x(x-1)$.}
\label{figure:WAN_x(x-1)_loss}
\end{figure}

Figure \ref{figure:DRM_x(x-1)_loss} shows the loss evolution for the DRM.  Here, we have a single minimization,  so the observed noisy behavior of the loss towards the end of the training is due to the optimizer.  Because of the lower complexity of the training, we obtain a better convergence performance than with WANs.

\begin{figure}[htbp]
\centering
\begin{subfigure}{\textwidth}
\centering
\includegraphics{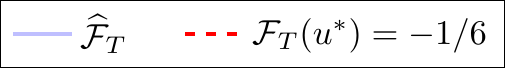}
\end{subfigure}\vskip 1em%
\centering
\begin{subfigure}[t]{0.5\textwidth}
\centering
\includegraphics{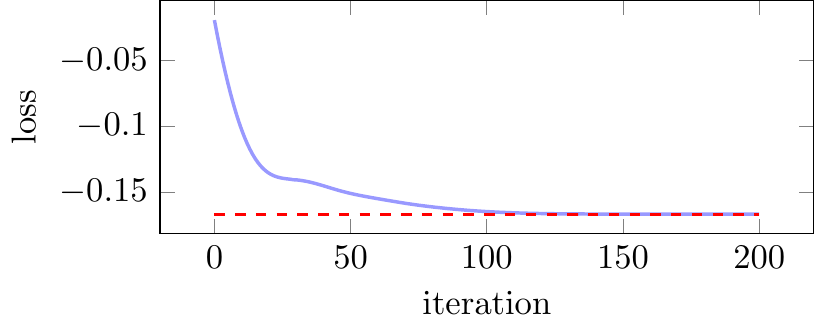}
\label{figure:DRM_loss_global}
\end{subfigure} \hskip 1em%
\begin{subfigure}[t]{0.3\textwidth}
\centering
\includegraphics{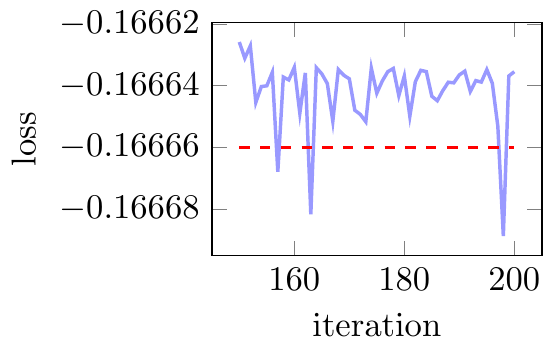}
\label{figure:DRM_loss_end}
\end{subfigure}
\caption{Loss evolution during the DRM training in problem \eqref{1DLap} with analytic solution $u^*=x(x-1)$.}
\label{figure:DRM_x(x-1)_loss}
\end{figure}

Figure \ref{figure:DDRM_x(x-1)_loss} shows the loss evolutions for the D$^2$RM.  In this method, we have a min-min optimization, where both losses are minimized alternatively.  At each iteration, we evaluate both $\widehat{\F}_{\tau_\n}$ and $\widehat{\L}_{u_\n}$, even if we are only optimizing with respect to one of them. We superimpose both losses, each one with its own scale (the left vertical axis corresponds to $\widehat{\F}_{\tau_\n}$ and the right vertical one to $\widehat{\L}_{u_\n}$).  Both losses exhibit a decreasing staircase shape with downward-sloping steps.  Jumps occur when the optimization is performed with respect to $\widehat{\F}_{\tau_\n}$, which suggests that $\widehat{\L}_{u_\n}$ takes longer to converge, as it depends on the convergence of $\widehat{\F}_{\tau_\n}$ (outer vs.  inner loop). 

\begin{figure}[htbp]
\centering
\begin{subfigure}{\textwidth}
\centering
\includegraphics{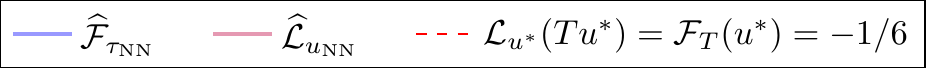}
\end{subfigure}\vskip 1em
\begin{subfigure}[t]{0.8\textwidth}
\centering
\includegraphics{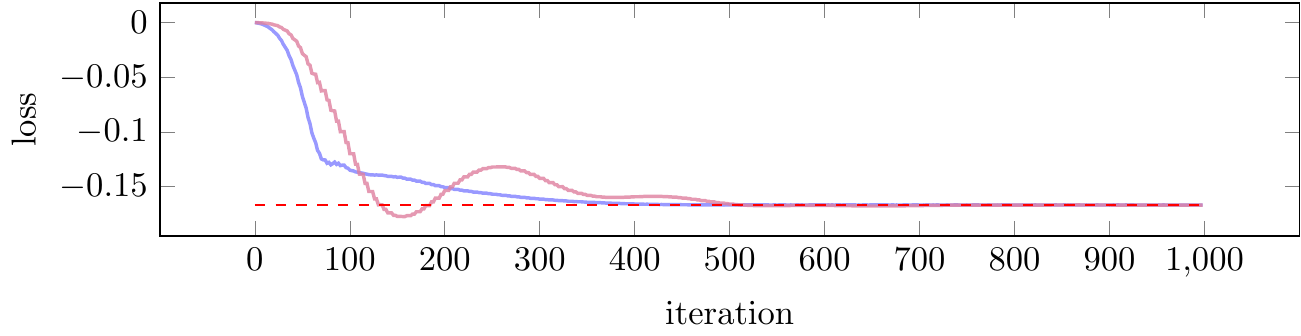}
\end{subfigure}\vskip 1em
\begin{subfigure}[t]{0.4\textwidth}
\centering
\includegraphics{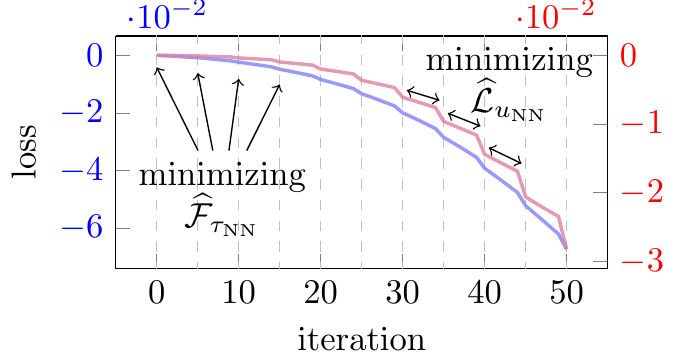}
\end{subfigure}\hskip 2em
\begin{subfigure}[t]{0.4\textwidth}
\centering
\includegraphics{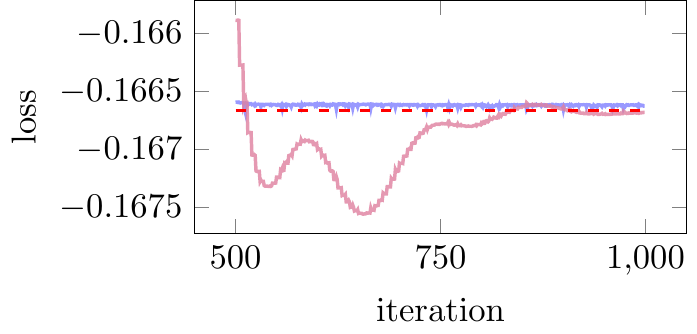}
\end{subfigure}
\caption{Loss evolution of D$^2$RM training for problem \eqref{1DLap} with analytic solution $u^*=x(x-1)$.}
\label{figure:DDRM_x(x-1)_loss}
\end{figure}

The relative errors of the trial network predictions\footnote{To approximate $\frac{\Vert u_{\n}-u\Vert_\U}{\Vert u\Vert_\U}\times 100$, we perform a composite intermediate-point rule with $10^4$ integration nodes for the numerator, and analytically calculate $\Vert u\Vert_\U = \sqrt{3}/3$ for the denominator.} at the end of the training are $3.12\%$, $0.99\%$,  and $1.31\%$ in WANs, the DRM, and the D$^2$RM, respectively.

\subsection{Comparison of WANs,  the DRM, and the D$^2$RM on smooth and singular problems.}\label{section:Comparison of performances by approaches}
Now,  we focus on the evolution of the relative error for the previous weak formulation of \eqref{1DLap}, but selecting the source so that the solution varies according to a parameter:
\begin{equation}
u_\alpha^* = x^{\alpha}(x-1)\in H^1_0(0,1), \qquad \alpha>1/2.
\end{equation}

\subsubsection{Without singularities: $\alpha\geq 1$}
We experiment individually for $\alpha \in\{2, 5,10\}$ with $5\mathord{,}000$ training  iterations for $u_\n$ in WANs, the DRM, and the D$^2$RM, which corresponds to a total of $25\mathord{,}000$ iterations for WANs and the D$^2$RM, when taking into account the iterations dedicated to the test maximizers. Table \ref{table:experiments2} displays the relative errors along different stages of the training, and Figure \ref{figure:xalpha(x-1)_predictions} shows the trial network predictions and error functions at the end of training.

\begin{table}[htbp]
\centering
\begin{tabular}{|c|c||c|c|c|c|c|}
\toprule
\multicolumn{2}{|c||}{\textbf{Training progress}} & $\mathbf{4}\bf{\%}$ & $\mathbf{20}\bf{\%}$ & $\mathbf{40}\bf{\%}$ & $\mathbf{60}\bf{\%}$ & $\mathbf{100}\bf{\%}$\\ \bottomrule
\toprule
\textbf{Method} & $\boldsymbol{\alpha}$ & \multicolumn{5}{c|}{$\frac{\Vert u_\n-u^*\Vert_\U}{\Vert u^*\Vert_\U}\times 100$}\\ \midrule
\multicolumn{1}{|l||}{\multirow{3}{*}{WANs}}     & $2$ &  $2.49\%$ & $3.43\%$  & $5.92\%$ & $7.40\%$ & $10.07\%$\\ \cline{2-7} 
\multicolumn{1}{|l||}{}                       & $5$ & $58.69\%$ & $41.78\%$ & $63.03\%$ & $74.06\%$ & $40.33\%$ \\ \cline{2-7} 
\multicolumn{1}{|l||}{}                       & $10$ & $93.15\%$ & $93.95\%$ & $89.04\%$ & $68.85\%$ & $367.61\%$ \\ \midrule
\multicolumn{1}{|l||}{\multirow{3}{*}{DRM}}  & $2$ & $3.40\%$ & $2.47\%$ & $1.17\%$ & $0.30\%$ & $0.23\%$ \\ \cline{2-7} 
\multicolumn{1}{|l||}{}                       & $5$ &  $50.50\%$&$8.53\%$&$2.83\%$&$2.27\%$&$1.59\%$\\ \cline{2-7} 
\multicolumn{1}{|l||}{}                       & $10$ &  $58.55\%$&$9.51\%$&$3.39\%$&$2.83\%$&$1.69\%$\\ \midrule
\multicolumn{1}{|l||}{\multirow{3}{*}{D$^2$RM}} & $2$ &  $3.27\%$&$1.84\%$&$0.31\%$&$0.27\%$&$0.54\%$\\ \cline{2-7} 
\multicolumn{1}{|l||}{}                       & $5$ & $59.13\%$&$12.93\%$&$2.92\%$&$2.52\%$&$1.56\%$\\ \cline{2-7} 
\multicolumn{1}{|l||}{}                       & $10$ & $82.31\%$ & $20.18\%$ & $6.24\%$& $2.68\%$& $2.60\%$\\ \bottomrule
\toprule
\textbf{Method} & $\boldsymbol{\alpha}$ & \multicolumn{5}{c|}{$\frac{\Vert \tau_\n(u_\n)- Tu^*\Vert_\V}{\Vert Tu^*\Vert_\V}\times 100$}\\ \midrule
\multicolumn{1}{|l||}{\multirow{3}{*}{D$^2$RM}} & $2$ &  $3.36\%$&$1.94\%$&$0.48\%$&$0.45\%$&$0.58\%$\\ \cline{2-7} 
\multicolumn{1}{|l||}{}                       & $5$ &  $59.13\%$&$12.93\%$&$2.90\%$&$2.11\%$&$1.55\%$\\ \cline{2-7} 
\multicolumn{1}{|l||}{}                       & $10$ &  $83.51\%$&$20.18\%$&$6.23\%$&$2.68\%$&$2.61\%$\\ \bottomrule
\end{tabular}%
\caption{Relative errors of $u_\n$ (in WANs, the DRM, and the D$^2$RM) and $\tau_\n(u_\n)$ (in the D$^2$RM) along different stages of the training progress in problem \eqref{1DLap} with analytic solution $u_\alpha^*=x^{\alpha}(x-1)$ and $\alpha\in\{2,5,10\}$. }
\label{table:experiments2}
\end{table}

\begin{figure}[htbp]
\centering
\begin{subfigure}[t]{\textwidth}
\centering
\includegraphics{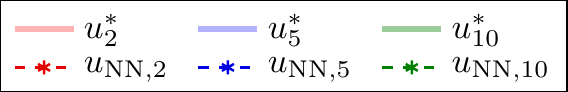}
\end{subfigure}\vskip 1em%
\begin{subfigure}[t]{0.31\textwidth}
\centering
\includegraphics{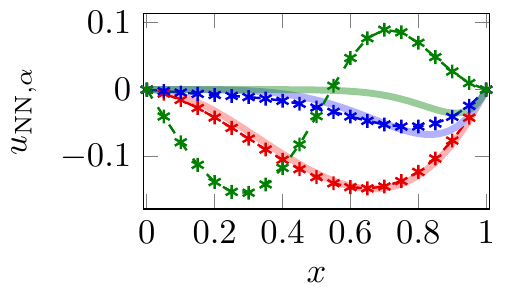}
\end{subfigure}\hskip 1em%
\begin{subfigure}[t]{0.31\textwidth}
\centering
\includegraphics{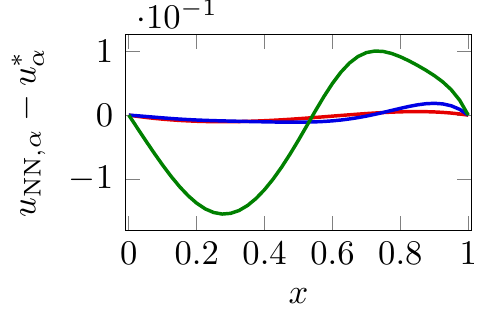}
\caption{With WANs}
\label{figure:WAN_xalpha(x-1)_25000}
\end{subfigure}\hskip 1em%
\begin{subfigure}[t]{0.31\textwidth}
\centering
\includegraphics{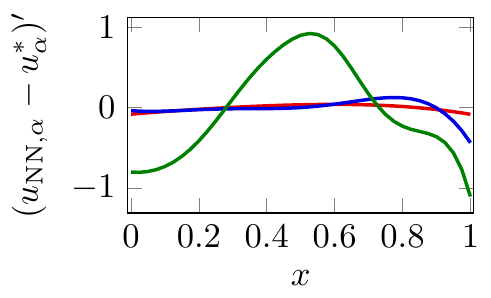}
\end{subfigure}\vskip 1em%
\begin{subfigure}[t]{0.31\textwidth}
\centering
\includegraphics{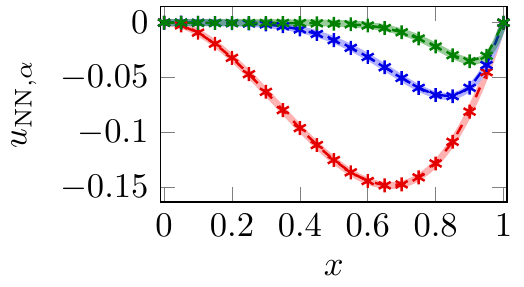}
\end{subfigure}\hskip 1em%
\begin{subfigure}[t]{0.31\textwidth}
\centering
\includegraphics{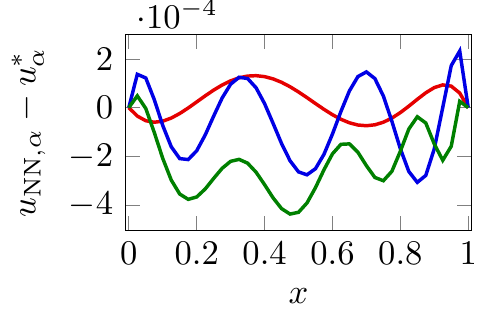}
\caption{With the DRM}
\label{figure:DRM_xalpha(x-1)_5000}
\end{subfigure}\hskip 1em%
\begin{subfigure}[t]{0.31\textwidth}
\centering
\includegraphics{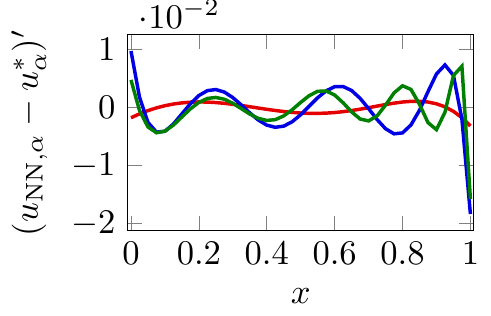}
\end{subfigure}\vskip 1em%
\begin{subfigure}[t]{0.31\textwidth}
\centering
\includegraphics{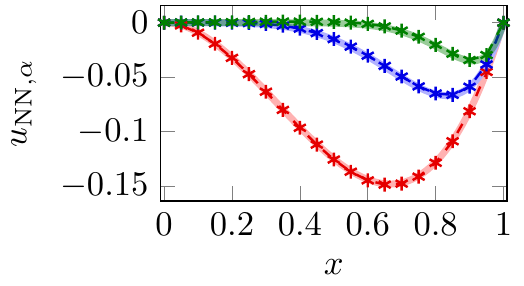}
\end{subfigure}\hskip 1em%
\begin{subfigure}[t]{0.31\textwidth}
\centering
\includegraphics{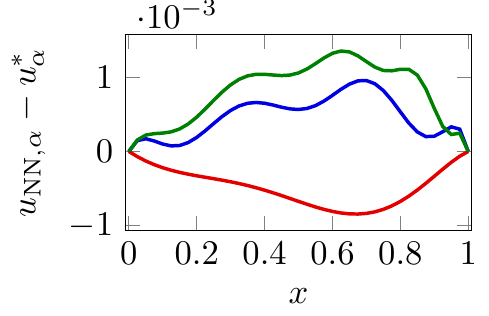}
\caption{With the D$^2$RM}
\label{figure:DDRM_xalpha(x-1)_25000}
\end{subfigure}\hskip 1em%
\begin{subfigure}[t]{0.31\textwidth}
\centering
\includegraphics{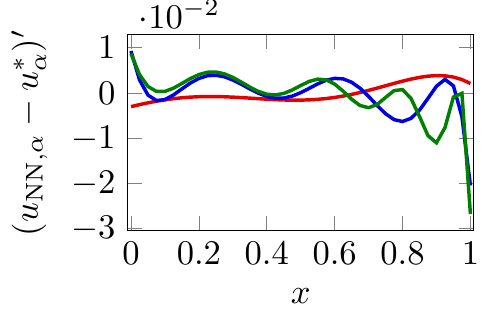}
\end{subfigure}
\caption{Trial network predictions and errors in WANs, the DRM, and the D$^2$RM in problem \eqref{1DLap} with analytic solution $u_\alpha^*=x^{\alpha}(x-1)$ for $\alpha\in\{2,5,10\}$. }
\label{figure:xalpha(x-1)_predictions}
\end{figure}

WANs show poor results in all the cases,  with a clear non-convergent tendency as the training progresses, possibly justified by the unstable behavior of the method at the continuous level. Thus, we discard the WANs for the remainder of the experiments and focus on the other two methods.

In the DRM and the D$^2$RM, we observe a decreasing behavior of the relative error during training.  We highlight the nearly identical behavior of the relative norm errors of $u_\n$ and $\tau_\n(u_\n)$ in the D$^2$RM,  which suggests that the D$^2$RM behaves as the DRM,  as desired. 

\subsubsection{With singularities: $1/2 < \alpha < 1$}
Here,  $(u^*_\alpha)'(x)\to-\infty$ as $x\to 0$, which suggests that a large portion of the integration nodes should concentrate on a neighborhood of zero to appropriately capture the explosive trend of the derivative at that point.  To this end,  we divide the batch as the union of two equal-sized different $\beta(a,b)$ samples:  one with $a=1=b$, and the other with $a=10^4$ and $b=1$. We individually experiment for $\alpha\in\{0.6, 0.7, 0.8\}$ with $100\mathord{,}000$ iterations for approximating $u_\n$, which implies $500\mathord{,}000$ for the D$^2$RM when taking into account the iterations dedicated to approximate the trial-to-test operator. Table \ref{table:experiments3} displays the record of the relative error estimates along different stages of the training,  and Figure \ref{figure:xalpha(x-1)_singular_predictions}  shows the $u_\n$ predictions and error functions at the end of the training.

\begin{table}[htbp]
\centering
\begin{tabular}{|c|c||c|c|c|c|c|}
\toprule
\multicolumn{2}{|c||}{\textbf{Training progress}} & $\mathbf{4}\bf{\%}$ & $\mathbf{20}\bf{\%}$ & $\mathbf{40}\bf{\%}$ & $\mathbf{60}\bf{\%}$ & $\mathbf{100}\bf{\%}$\\ \bottomrule
\toprule
\textbf{Method} & $\boldsymbol{\alpha}$ & \multicolumn{5}{c|}{$\frac{\Vert u_\n-u^*\Vert_\U}{\Vert u^*\Vert_\U}\times 100$} \\ \midrule
\multicolumn{1}{|l||}{\multirow{3}{*}{DRM}}  & $0.6$ & $42.23\%$ & $30.29\%$ & $26.91\%$ & $25.34\%$ & $23.84\%$\\ \cline{2-7} 
\multicolumn{1}{|l||}{}                       & $0.7$ &  $18.47\%$ &$9.49\%$ &$7.42\%$ &$6.65\%$ &$5.95\%$\\ \cline{2-7} 
\multicolumn{1}{|l||}{}                       & $0.8$ &  $8.78\%$ &$3.50\%$ &$2.46\%$ &$2.10\%$ &$1.81\%$\\ \midrule
\multicolumn{1}{|l||}{\multirow{3}{*}{D$^2$RM}} & $0.6$ &  $41.67\%$ &$30.25\%$ &$27.05\%$ &$25.76\%$ &$24.40\%$\\ \cline{2-7} 
\multicolumn{1}{|l||}{}                       & $0.7$ & $14.12\%$ &$9.99\%$ &$7.91\%$ &$6.98\%$ &$6.15\%$ \\ \cline{2-7} 
\multicolumn{1}{|l||}{}                       & $0.8$ & $6.00\%$ & $3.62\%$ & $2.78\%$ & $2.55\%$ & $2.13\%$\\ \bottomrule
\toprule
\textbf{Method} & $\boldsymbol{\alpha}$ & \multicolumn{5}{c|}{$\frac{\Vert \tau_\n(u_\n)-Tu^*\Vert_\V}{\Vert Tu^*\Vert_\V}\times 100$} \\ \midrule
\multicolumn{1}{|l||}{\multirow{3}{*}{D$^2$RM}} & $0.6$ &  $41.68\%$ &$30.25\%$ &$27.05\%$ &$25.77\%$ &$24.40\%$ \\ \cline{2-7} 
\multicolumn{1}{|l||}{}                       & $0.7$ &  $14.12\%$ &$10.00\%$ &$7.91\%$ &$6.98\%$ &$6.15\%$ \\ \cline{2-7} 
\multicolumn{1}{|l||}{}                       & $0.8$ &  $6.01\%$ &$3.62\%$ &$2.78\%$ &$2.54\%$ &$2.13\%$\\ \bottomrule
\end{tabular}%
\caption{Relative errors of $u_\n$ (in the DRM, and the D$^2$RM) and $\tau_\n(u_\n)$ (in the D$^2$RM) along different stages of the training progress in problem \eqref{1DLap} with analytic solution $u_\alpha^*=x^{\alpha}(x-1)$ and $\alpha\in\{0.6, 0.7, 0.8\}$. }
\label{table:experiments3}
\end{table}

\begin{figure}[htbp]
\centering
\begin{subfigure}[t]{\textwidth}
\centering
\includegraphics{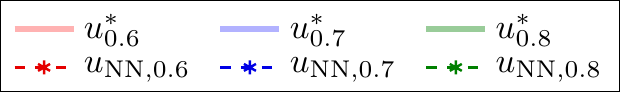}
\end{subfigure}\vskip 1em%
\begin{subfigure}[t]{0.35\textwidth}
\centering
\includegraphics{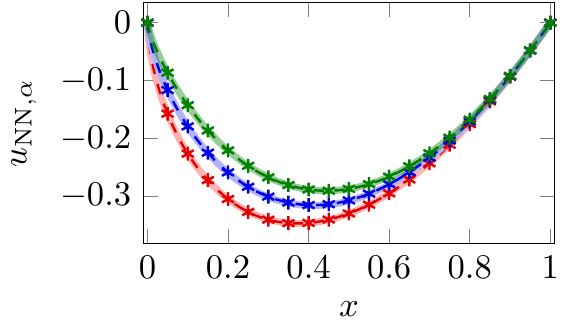}
\end{subfigure}\hskip 1em%
\begin{subfigure}[t]{0.35\textwidth}
\centering
\includegraphics{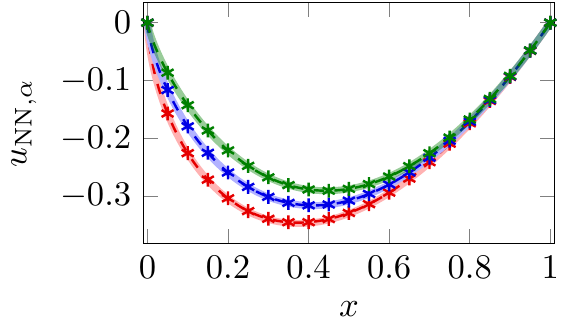}
\end{subfigure}\vskip 1em%
\begin{subfigure}[t]{0.35\textwidth}
\centering
\includegraphics{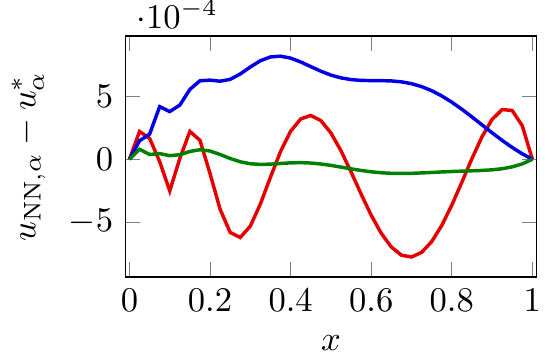}
\end{subfigure}\hskip 1em%
\begin{subfigure}[t]{0.35\textwidth}
\centering
\includegraphics{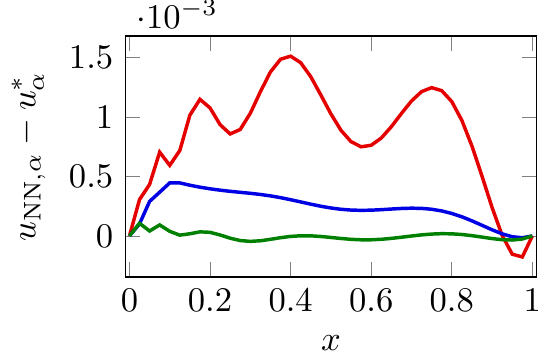}
\end{subfigure}\vskip 1 em%
\begin{subfigure}[t]{0.35\textwidth}
\centering
\includegraphics{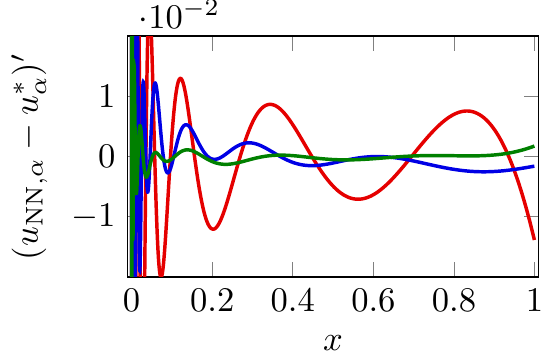}
\end{subfigure}\hskip 1em%
\begin{subfigure}[t]{0.35\textwidth}
\centering
\includegraphics{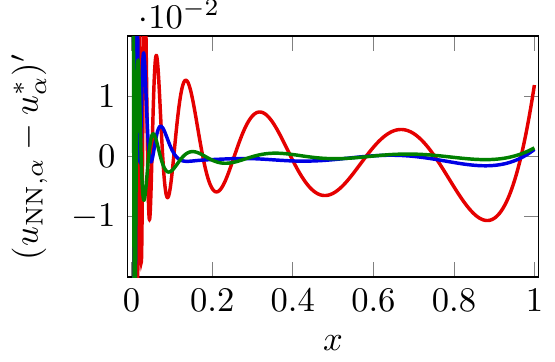}
\end{subfigure}\vskip 1 em%
\begin{subfigure}[t]{0.35\textwidth}
\centering
\includegraphics{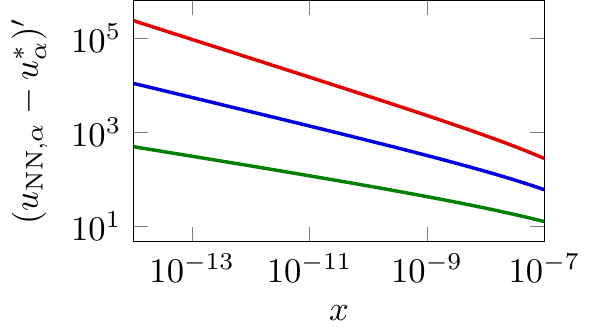}
\caption{With the DRM}
\end{subfigure}\hskip 1em%
\begin{subfigure}[t]{0.35\textwidth}
\centering
\includegraphics{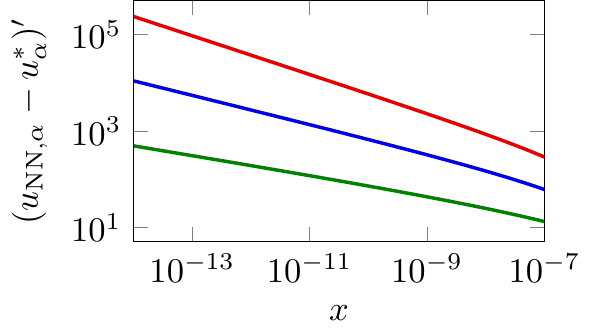}
\caption{With the D$^2$RM}
\end{subfigure}
\caption{Trial network predictions and error functions in the DRM and the D$^2$RM in problem \eqref{1DLap} with analytic solution $u_\alpha^*=x^{\alpha}(x-1)$ for $\alpha\in\{0.6,0.7,0.8\}$.  The last row is an augmented version of the third one in a neighborhood of zero.}
\label{figure:xalpha(x-1)_singular_predictions}
\end{figure}

The DRM and D$^2$RM decrease the relative errors at similar rates, slowing down notably from the $40\%$ of the training progress onwards.  For $\alpha \in\{ 0.7,  0.8\}$, we achieve final relative errors below $7.5\%$ at the end of training, with loss predictions around $-0.3632$ and $-0.2562$ at the end of the training where the optimal values are around $-0.3646$ and $-0.2564$, respectively.  For $\alpha = 0.6$, the relative error is above $20\%$,  with a final loss prediction around $-0.6429$ where the optimal value is around $-0.6818$.  We repeated several experiments in this case modifying the $a$ parameter of the beta distribution with the intention of concentrating more points in the singularity, but without success. \pagebreak

\subsection{Pure diffusion equation with Dirac delta source}\label{section:Diffusion equation with a Dirac delta source}
We now select the source $l=4 \delta_{1/2}\in H^{-1}(0,1)\setminus L^2(0,1)$, where $\delta_{1/2}$ denotes the Dirac delta at $1/2$,  i.e.,  $\delta_{1/2}(v)=v(1/2)$ for all $v\in H^1_0(0,1)$. The analytic solution is
 \begin{equation}\label{equation:corner_solution}
 u^*=\begin{cases}2x, &\text{if }x<1/2,\\
 2(1-x), &\text{if } x>1/2.
 \end{cases}
 \end{equation} 

We train $u_\n$ for $20\mathord{,}000$ iterations selecting $a=1=b$ and $a=10=b$ in the $\beta(a,b)$ probability distribution for the two-sampled batch.  Figure \ref{figure:corner_u} shows the trial network predictions at three different phases of the training progress, and Figure \ref{figure:corner_error_u} shows the error functions and their derivatives at those checkpoints.  Table \ref{table:experiments4} displays a record of the relative errors.  In this occasion,  the D$^2$RM shows higher errors than the DRM.  

\begin{figure}[htbp]
\centering
\begin{subfigure}[t]{\textwidth}
\centering
\includegraphics{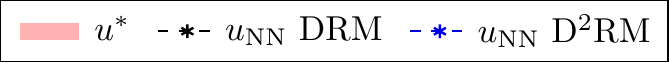}
\end{subfigure}\vskip 1em
\begin{subfigure}[t]{0.30\textwidth}
\centering
\includegraphics{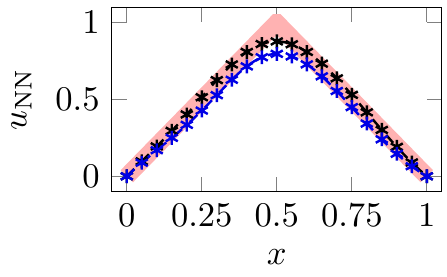}
\caption{At $4\%$ of the training progess}
\label{figure:DeepRitz_corner_prediction}
\end{subfigure}\hskip 1em
\begin{subfigure}[t]{0.30\textwidth}
\centering
\includegraphics{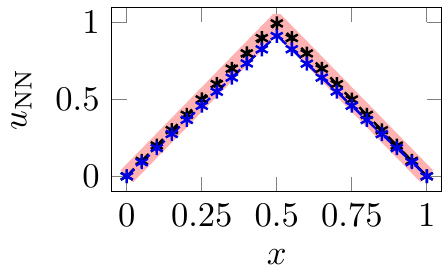}
\caption{At $40\%$ of the training progess}
\label{figure:DeepDRitz_corner_prediction}
\end{subfigure}\hskip 1em
\begin{subfigure}[t]{0.30\textwidth}
\centering
\includegraphics{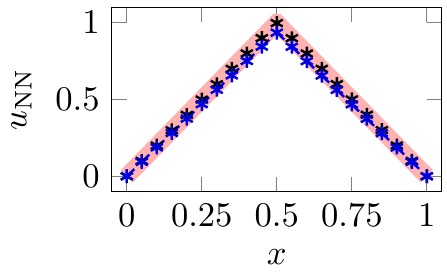}
\caption{At $100\%$ of the training progess}
\label{figure:DeepDRitz_corner_prediction}
\end{subfigure}
\caption{Trial network predictions for the DRM and the D$^2$RM in problem \eqref{1DLap} with analytic solution \eqref{equation:corner_solution}. }
\label{figure:corner_u}
\end{figure}

\begin{figure}[htbp]
\centering
\begin{subfigure}[t]{\textwidth}
\centering
\includegraphics{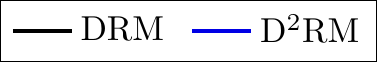}
\end{subfigure}\vskip 1em%
\begin{subfigure}[t]{0.30\textwidth}
\centering
\includegraphics{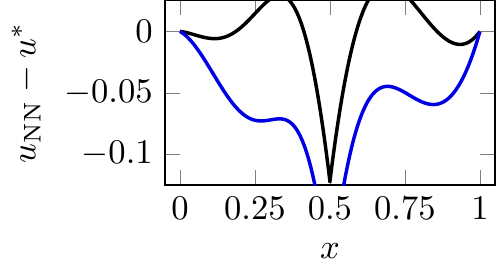}
\end{subfigure}\hskip 1em%
\begin{subfigure}[t]{0.30\textwidth}
\centering
\includegraphics{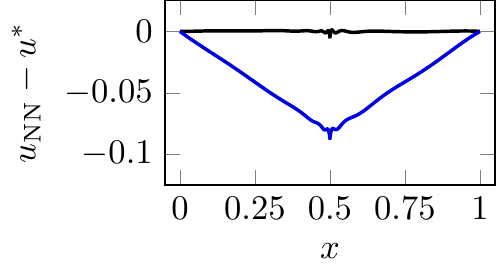}
\end{subfigure}\hskip 1em%
\begin{subfigure}[t]{0.30\textwidth}
\centering
\includegraphics{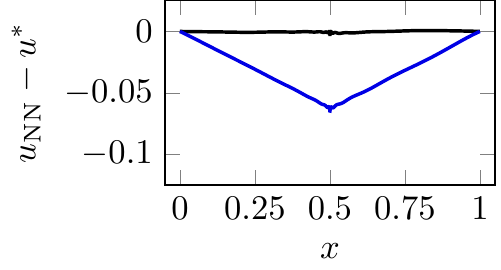}
\end{subfigure}\vskip 1em%
\begin{subfigure}[t]{0.30\textwidth}
\centering
\includegraphics{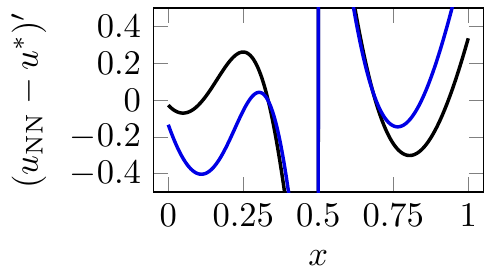}
\caption{At $4\%$ of the training progess}
\label{figure:corner_derror1}
\end{subfigure}\hskip 1em%
\begin{subfigure}[t]{0.30\textwidth}
\centering
\includegraphics{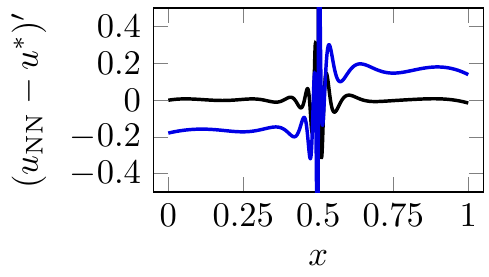}
\caption{At $40\%$ of the training progess}
\label{figure:corner_derror2}
\end{subfigure}\hskip 1em%
\begin{subfigure}[t]{0.30\textwidth}
\centering
\includegraphics{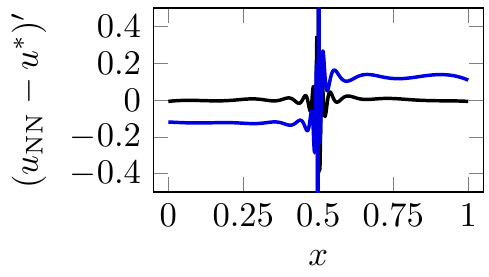}
\caption{At $100\%$ of the training progess}
\label{figure:corner_derror3}
\end{subfigure}
\caption{Trial error functions in problem \eqref{1DLap} with analytic solution \eqref{equation:corner_solution} at different stages of the training progress.}
\label{figure:corner_error_u}
\end{figure}

\begin{table}[htbp]
\centering
\begin{tabular}{|c||c|c|c|c|c|}
\toprule
\textbf{Training progress} & $\mathbf{4}\bf{\%}$ & $\mathbf{20}\bf{\%}$ & $\mathbf{40}\bf{\%}$ & $\mathbf{60}\bf{\%}$ & $\mathbf{100}\bf{\%}$\\ \bottomrule
\toprule
\textbf{Method} & \multicolumn{5}{c|}{$\frac{\Vert u_\n-u^*\Vert_\U}{\Vert u^*\Vert_\U}\times 100$}\\ \midrule
DRM & $32.32\%$ & $10.10\%$ & $6.22\%$ & $5.08\%$ & $4.34\%$\\ \midrule
D$^2$RM & $32.43\%$&$13.95\%$&$10.87\%$&$9.41\%$&$7.95\%$\\ \bottomrule
\toprule
\textbf{Method} & \multicolumn{5}{c|}{$\frac{\Vert \tau_\n(u_\n)-Tu^*\Vert_\V}{\Vert Tu^*\Vert_\V}\times 100$}\\ \midrule
D$^2$RM & $32.68\%$&$14.09\%$&$10.84\%$&$9.53\%$&$8.16\%$\\ \bottomrule
\end{tabular}%
\caption{Relative errors of $u_\n$ (in the DRM, and the D$^2$RM) and $\tau_\n(u_\n)$ (in the D$^2$RM) along different stages of the training progress in problem \eqref{1DLap} with analytic solution \eqref{equation:corner_solution}. }
\label{table:experiments4}
\end{table}

Figure \ref{figure:corner_u_loss} shows the losses evolution for the DRM and the D$^2$RM.  We observe that in the DRM and the D$^2$RM, the loss related to the trial networks attain thresholds below $-1.98$, where $-2$ is the optimum at the continuous level,  but the loss related to the optimal test networks remains much higher. 

\begin{figure}[htbp]
\centering
\begin{subfigure}[t]{\textwidth}
\centering
\includegraphics{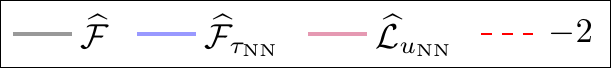}
\end{subfigure}\vskip 1em%
\begin{subfigure}[t]{0.30\textwidth}
\centering
\includegraphics{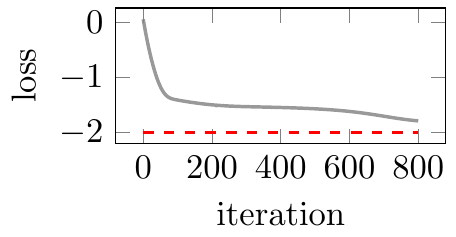}
\end{subfigure}\hskip 1em%
\begin{subfigure}[t]{0.30\textwidth}
\centering
\includegraphics{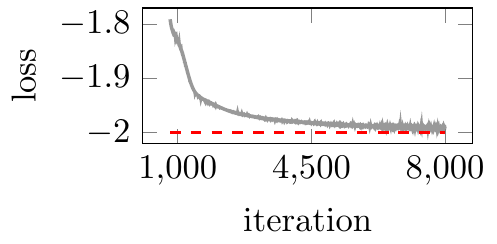}
\end{subfigure}\hskip 1em%
\begin{subfigure}[t]{0.30\textwidth}
\centering
\includegraphics{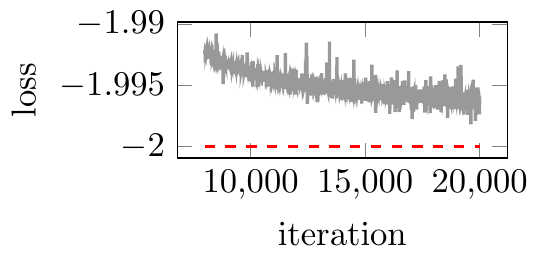}
\end{subfigure}\vskip 1em%
\begin{subfigure}[t]{0.30\textwidth}
\centering
\includegraphics{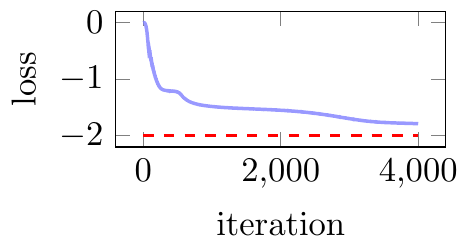}
\end{subfigure}\hskip 1em%
\begin{subfigure}[t]{0.30\textwidth}
\centering
\includegraphics{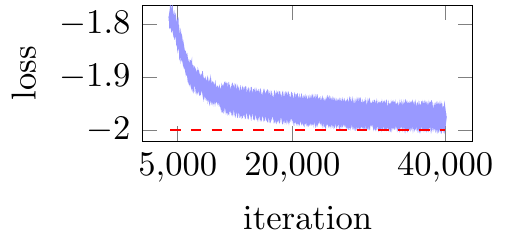}
\end{subfigure}\hskip 1em%
\begin{subfigure}[t]{0.30\textwidth}
\centering
\includegraphics{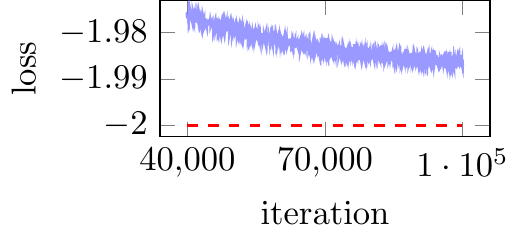}
\end{subfigure}\vskip 1em%
\begin{subfigure}[t]{0.30\textwidth}
\centering
\includegraphics{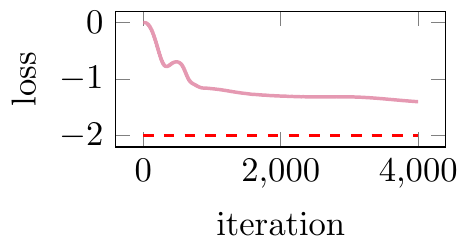}
\end{subfigure}\hskip 1em%
\begin{subfigure}[t]{0.30\textwidth}
\centering
\includegraphics{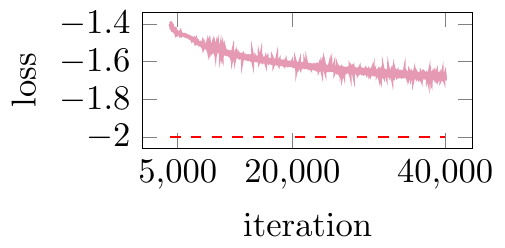}
\end{subfigure}\hskip 1em%
\begin{subfigure}[t]{0.30\textwidth}
\centering
\includegraphics{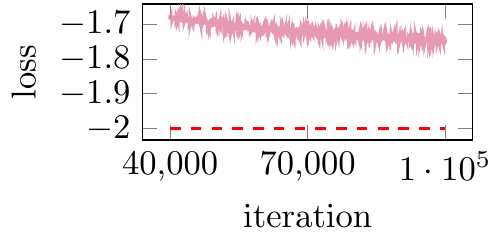}
\end{subfigure}
\caption{Loss functions evolution in the DRM and the D$^2$RM in problem \eqref{1DLap} with exact solution \eqref{equation:corner_solution}. }
\label{figure:corner_u_loss}
\end{figure}

\subsection{Pure convection equation in ultraweak form}\label{section:Advection equation in ultraweak formulation}
We consider the spatial convection equation
\begin{equation}\label{equation:advection}
\begin{cases}
u'=\delta_{1/2}, &\text{in } (0,1),\\
u(0)=0.
\end{cases}
\end{equation} Here, the (ultra)weak formulation is appropriate because $\delta_{1/2}$ does not belong to $L^2(0,1)$.  Integration by parts yields:
\begin{equation}
b(u,v):=-\int_0^1 uv', \quad l(v):=v(1/2),\qquad u\in\U, v\in\V,
\end{equation} with $\U:=L^2(0,1)$ and $\V:=H^1_{0)}:=\{v\in H^1(0,1):v(1)=0\}$.  The trial-to-test operator is no longer the identity and has the following integral form \cite{gopalakrishnan2013five, munoz2021equivalence}:
\begin{equation}\label{trial-to-test convection}
(Tu)(x)=\int_x^1 u(s) ds, \qquad u\in L^2(0,1).
\end{equation} For the exact solution and its corresponding optimal test function, we have
\begin{equation}\label{equation:advection_solution}
u^*=\begin{cases}
0, &\text{if } 0<x<1/2,\\
1, &\text{if } 1/2<x<1,
\end{cases} \in L^2(0,1),
\qquad Tu^*=\begin{cases}
1/2 &\text{if } 0<x<1/2,\\
1-x, &\text{if } 1/2<x<1,
\end{cases} \in H^1_{0)}(0,1).
\end{equation} 

In our context of NNs,  considering $T$ as an available operator is challenging,  so we discard employing the GDRM in here, and alternatively employ the (DRM$)'$ and the D$^2$RM:

\begin{itemize}
\item \textbf{Adjoint Deep Ritz Method.} Following Section \ref{Generalized Ritz method},  item (c), we minimize $\F'$ to find an approximation to the optimal test function of the trial solution, i.e., 
\begin{equation}
Tu^* \approx \arg\min_{v_\n\in\V_\n} \F'(v_\n). 
\end{equation} Later, we post-process the above minimizer by applying the available adjoint operator $A'=-d/dx$.  

We perform $50\mathord{,}000$ training iterations.  Table \ref{table:experiments5} records the evolution of the approximated relative norm errors $\frac{\Vert v_\n-Tu\Vert_\V}{\Vert Tu\Vert_\V}$ and $\frac{\Vert A' v_\n-u \Vert_\U}{\Vert u\Vert_\U}$ along the training progress. Figure \ref{figure:Adjoint DRM_break} shows the predictions and errors of both $v_\n$ and $A' v_\n$ along different stages of the training. 

\begin{table}[htbp]
\centering
\begin{tabular}{|c|c|c|c|c|c|}
\toprule
\textbf{Training progress} & $\mathbf{4}\bf{\%}$ & $\mathbf{20}\bf{\%}$ & $\mathbf{40}\bf{\%}$ &$\mathbf{60}\bf{\%}$ & $\mathbf{100}\bf{\%}$\\
\toprule
$\frac{\Vert v_\n-Tu^*\Vert_\V}{\Vert Tu^*\Vert_\V}\times 100$  & $19.03\%$ & $4.58\%$ & $3.55\%$ & $3.20\%$ & $2.83\%$ \\ \midrule
$\frac{\Vert A' v_\n-u^*\Vert_\U}{\Vert u^*\Vert_\U}\times 100$  & $21.95\%$ & $5.28\%$ & $4.10\%$ & $3.69\%$ & $3.27\%$ \\ \bottomrule
\end{tabular}%
\caption{Relative errors of $v_\n$ and $A' v_\n$ in problem \eqref{equation:advection} with exact trial and optimal test solutions \eqref{equation:advection_solution} along different stages of the training progress.}
\label{table:experiments5}
\end{table}

\begin{figure}[htbp]
\centering
\begin{subfigure}[t]{\textwidth}
\centering
\includegraphics{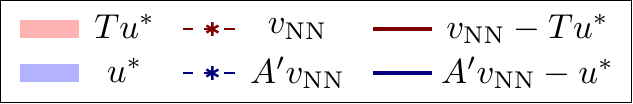}
\end{subfigure}\vskip 1em%
\begin{subfigure}[t]{0.30\textwidth}
\centering
\includegraphics{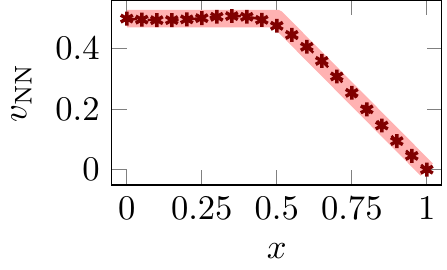}
\end{subfigure}\hskip 1em%
\begin{subfigure}[t]{0.30\textwidth}
\centering
\includegraphics{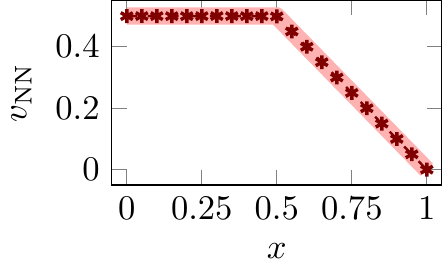}
\end{subfigure}\hskip 1em%
\begin{subfigure}[t]{0.30\textwidth}
\centering
\includegraphics{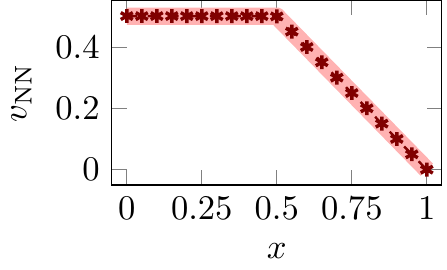}
\end{subfigure}\vskip 1em%
\begin{subfigure}[t]{0.30\textwidth}
\centering
\includegraphics{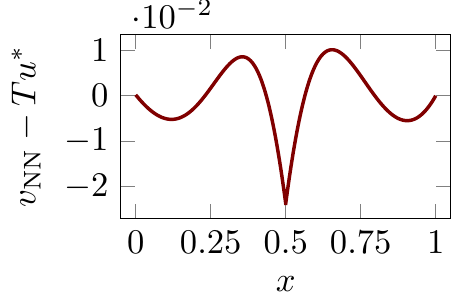}
\end{subfigure}\hskip 1em%
\begin{subfigure}[t]{0.30\textwidth}
\centering
\includegraphics{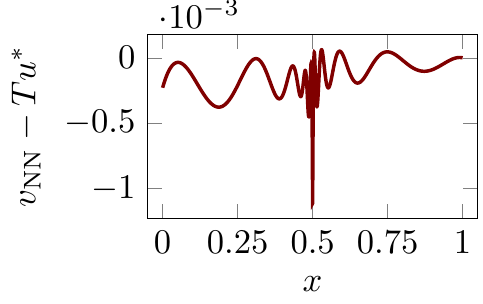}
\end{subfigure}\hskip 1em%
\begin{subfigure}[t]{0.30\textwidth}
\centering
\includegraphics{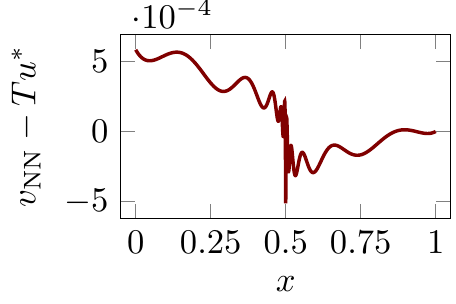}
\end{subfigure}\vskip 1em%
\begin{subfigure}[t]{0.30\textwidth}
\centering
\includegraphics{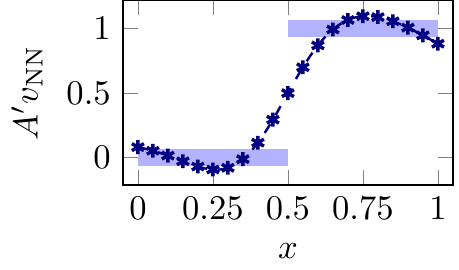}
\end{subfigure}\hskip 1em%
\begin{subfigure}[t]{0.30\textwidth}
\centering
\includegraphics{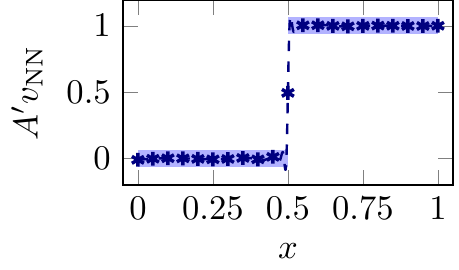}
\end{subfigure}\hskip 1em%
\begin{subfigure}[t]{0.30\textwidth}
\centering
\includegraphics{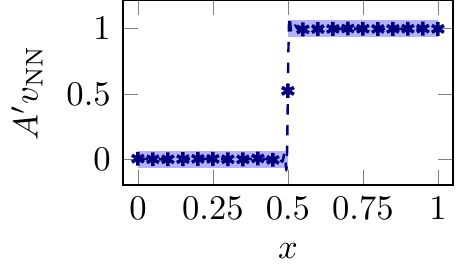}
\end{subfigure}\vskip 1em%
\begin{subfigure}[t]{0.30\textwidth}
\centering
\includegraphics{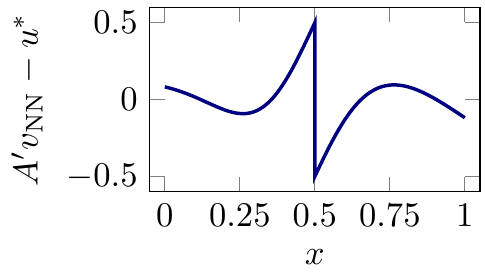}
\caption{At $4\%$ of the training progress}
\end{subfigure}\hskip 1em%
\begin{subfigure}[t]{0.30\textwidth}
\centering
\includegraphics{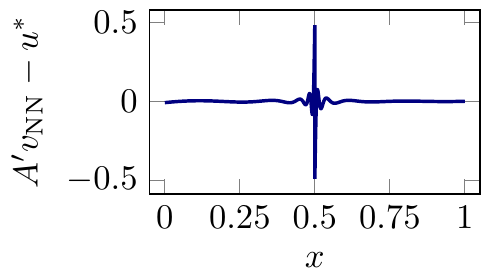}
\caption{At $40\%$ of the training progress}
\end{subfigure}\hskip 1em%
\begin{subfigure}[t]{0.30\textwidth}
\centering
\includegraphics{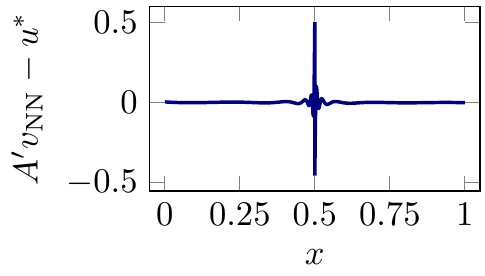}
\caption{At $100\%$ of the training progress}
\end{subfigure}
\caption{$v_\n$ network predictions,  post-processings $A' v_\n$,  and corresponding error functions for the (DRM$)'$ in problem \eqref{equation:advection} with exact trial solution \eqref{equation:advection_solution} along different stages of the training progress.}
\label{figure:Adjoint DRM_break}
\end{figure}

\item  \textbf{Deep Double Ritz Method.} We impose the outflow-boundary condition only to $\tau_\n$, letting $u_\n$ be boundary-free,  and perform $50\mathord{,}000$ iterations for $u_\n$ and increase the number of inner-loop iterations from four to nine (i.e., a total of $500\mathord{,}000$ iterations if we account those dedicated to the test maximizer).  Table \ref{table:experiments6} records the relative errors for $u_\n$ and $\tau_\n(u_\n)$ along different stages of the training. Figure \ref{figure:DDRM_break} shows the trial and optimal test network predictions and error functions at the end of the training.

\begin{table}[htbp]
\centering
\begin{tabular}{|c|c|c|c|c|c|}
\toprule
\textbf{Training progress} & $\mathbf{4}\bf{\%}$ & $\mathbf{20}\bf{\%}$ & $\mathbf{40}\bf{\%}$ &$\mathbf{60}\bf{\%}$ & $\mathbf{100}\bf{\%}$\\
\toprule
$\frac{\Vert u_\n - u^*\Vert_\U}{\Vert u^*\Vert_\U}$ & $42.05\%$ & $29.69\%$ & $18.42\%$ & $11.20\%$ & $8.93\%$ \\ \midrule
$\frac{\Vert \tau_\n(u_\n) - Tu^*\Vert_\V}{\Vert Tu^*\Vert_\V}$ & $37.27\%$ & $28.25\%$ & $19.02\%$ & $9.55\%$ & $6.13\%$ \\ \bottomrule
\end{tabular}%
\caption{Relative errors of $u_\n$ and $\tau_n u_\n$, respectively,  along different stages of the training progress in problem \eqref{equation:advection} with exact trial and optimal test solutions \eqref{equation:advection_solution}. }
\label{table:experiments6}
\end{table}

\begin{figure}[htbp]
\centering
\begin{subfigure}[t]{\textwidth}
\centering
\includegraphics{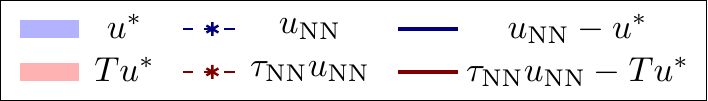}
\end{subfigure}\vskip 1em%
\begin{subfigure}[t]{0.30\textwidth}
\centering
\includegraphics{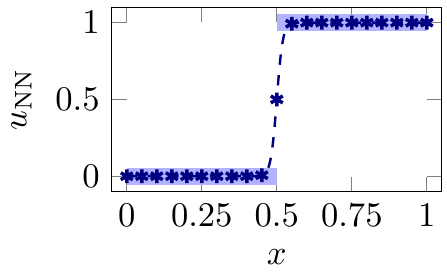}
\end{subfigure}\hskip 1em%
\begin{subfigure}[t]{0.30\textwidth}
\centering
\includegraphics{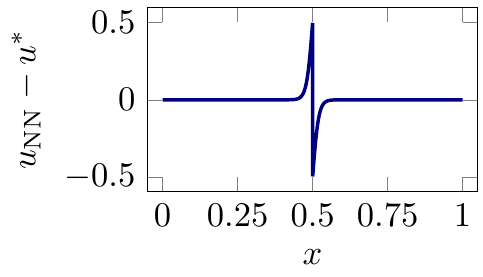}
\end{subfigure}\vskip 1em%
\begin{subfigure}[t]{0.30\textwidth}
\centering
\includegraphics{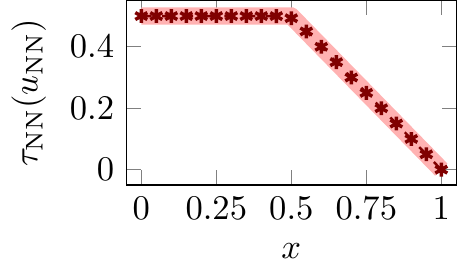}
\end{subfigure}\hskip 1em%
\begin{subfigure}[t]{0.30\textwidth}
\centering
\includegraphics{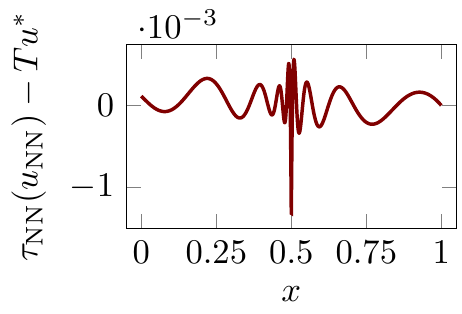}
\end{subfigure}\hskip 1em%
\begin{subfigure}[t]{0.30\textwidth}
\centering
\includegraphics{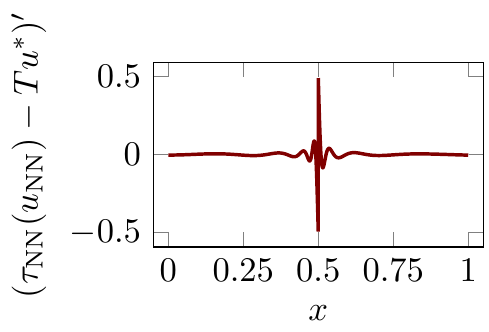}
\end{subfigure}
\caption{$u_\n$ and $\tau_\n(u_\n)$ predictions,  errors, and derivative of the errors for the D$^2$RM in problem \eqref{equation:advection} with exact trial and optimal test solutions \eqref{equation:advection_solution}.}
\label{figure:DDRM_break}
\end{figure}

\end{itemize}

\pagebreak

\review{
\subsection{2D pure convection in strong variational form}\label{2D section}

Let
\begin{equation} \label{equation:advection2D}
\begin{cases}
\displaystyle \frac{\partial u}{\partial x} + \frac{\partial u}{\partial y} =k\pi \sin\left(k\pi (x+y)\right), &\text{in } \Omega=(0,1)\times(0,1),\\
u(x,0)=u(0,y)=0,& 0\leq x,y\leq 1,
\end{cases}
\end{equation} with $k=3/2$, and consider its strong variational formulation, i.e., 
\begin{equation}
b(u,v):=\int_{\Omega} \left(\frac{\partial u}{\partial x} + \frac{\partial u}{\partial y}\right)v, \quad l(v):=k\pi \int_{\Omega}  \sin(k\pi (x+y)) v,\quad u\in\U, v\in\V,
\end{equation} with $\U:=\{u\in L^2(\Omega): \partial u/\partial x + \partial u/\partial y \in L^2(\Omega), u(x,0)=u(0,y)=0, 0\leq x,y\leq 1 \}$ and $\V:=L^2(\Omega)$. Its analytic solution is $u^*=\sin(3\pi x/2)\sin(3\pi y/2)$, and the trial-to-test operator is the PDE operator.

We perform $200\mathord{,}000$ iterations in the D$^2$RM with a training regime of nine iterations in the inner loop for each iteration in the outer loop.  Exceptionally, we run the inner loop for $2\mathord{,}000$ iterations before the first outer-loop iteration.  For integration,  we consider $50$ nodes on each axis (i.e., $250$ integration points on the entire domain due to the cartesian-product structure). We increase the NN architecture to three layers of $50$-neuron width.

Figure \ref{figure:DDRM_2D} shows the trial and optimal test predictions with corresponding error functions at the end of the training. The resulting relative errors are $3.62\%$ and $1.70\%$ for $u_\n$ and $\tau_\n(u_\n)$, respectively.

\begin{figure}[htbp]
\centering
\begin{subfigure}[t]{0.45\textwidth}
\centering
\includegraphics{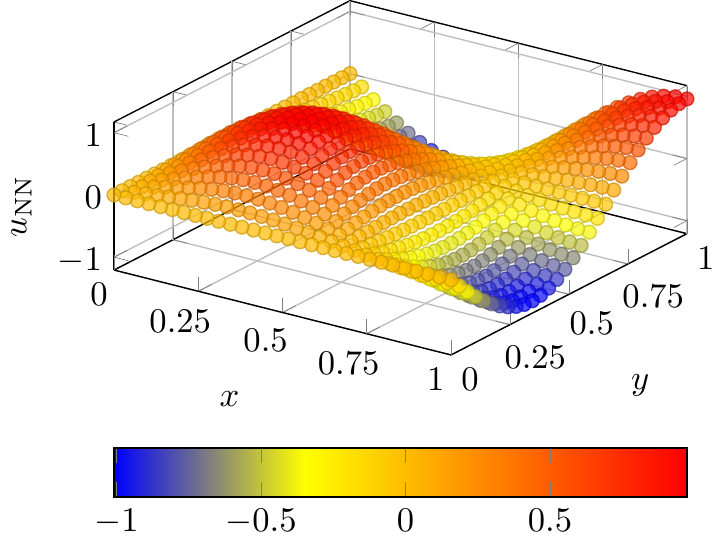}
\end{subfigure}\hskip 1em%
\begin{subfigure}[t]{0.45\textwidth}
\centering
\includegraphics{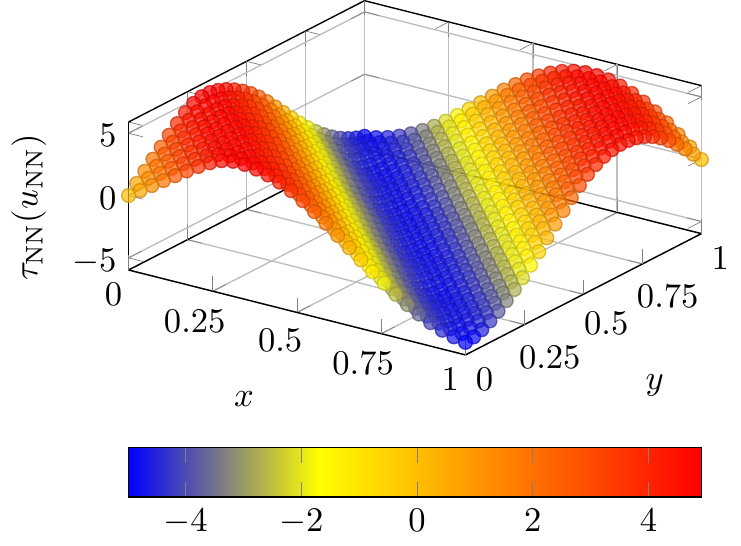}
\end{subfigure}\vskip 1em%
\begin{subfigure}[t]{0.45\textwidth}
\centering
\includegraphics{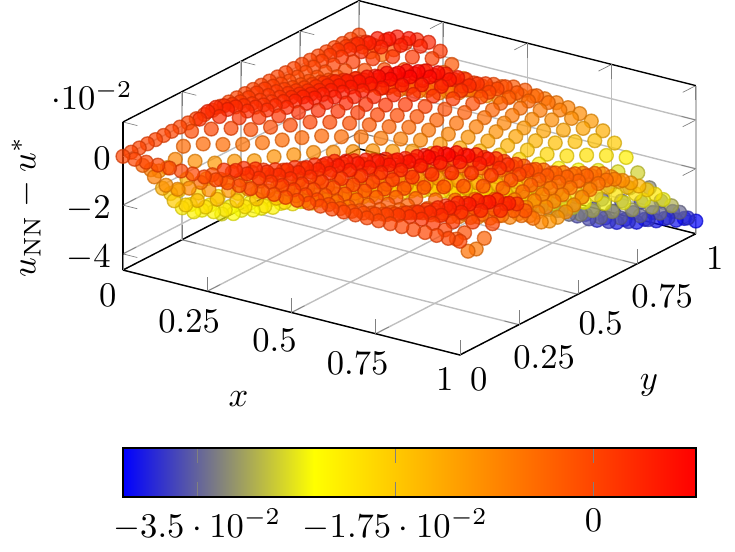}
\end{subfigure}\hskip 1em%
\begin{subfigure}[t]{0.45\textwidth}
\centering
\includegraphics{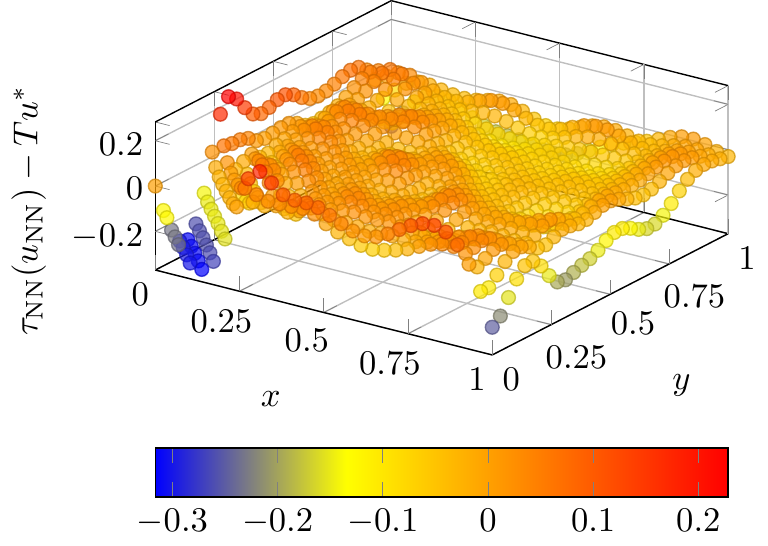}
\end{subfigure}\vskip 1em%
\begin{subfigure}[t]{0.45\textwidth}
\centering
\includegraphics{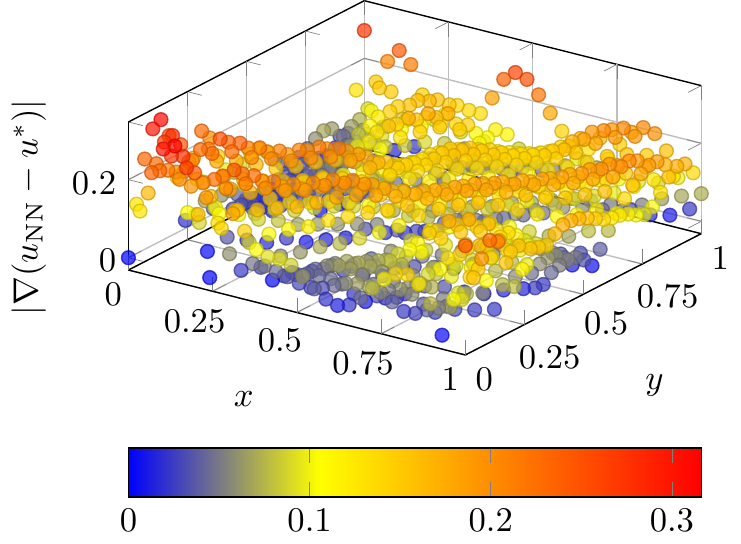}
\end{subfigure}
\caption{$u_\n$ and $\tau_\n(u_\n)$ predictions and errors for the D$^2$RM in problem \eqref{equation:advection2D} with exact trial and optimal test solutions $u^*=\sin(3\pi x/2)\sin(3\pi y/2)$ and $Tu^*=3\pi/2 \sin(3\pi x/2)\sin(3\pi y/2)$, respectively.}
\label{figure:DDRM_2D}
\end{figure}
}

\pagebreak

\section{Conclusions and future work}
\label{section:Conclusions and future work}

We studied the problem of residual minimization for solving PDEs using Neural Networks. First, we considered a min-max saddle point problem coming from the definition of the dual norm as a maximum over the test space. This method turned out numerically unstable because of the non-Lipschitz continuity of the test maximizers. To overcome this, we rewrote the general residual minimization as a minimization of the Ritz functional employing optimal test functions.  To carry out this alternative minimization while computing the optimal test functions over general problems,  we proposed a Deep Double Ritz Method (D$^2$RM) that combines two nested Ritz optimization loops. This novel method constructs local approximations of the trial-to-test operator to express the optimal test functions as dependent on the trial functions.  By doing this, NNs allowed us to easily combine the two nested Ritz problems in a way that is difficult to treat with traditional numerical methods.

We tested the D$^2$RM in several smooth and singular problems. Numerical results illustrate the advantages and limitations of the proposed methods. Among the advantages,  we encounter the good approximation capabilities of NNs, the generality of the proposed method, and its implementation to different linear PDEs, variational formulations, and spatial dimensions. As main limitations, we face two common difficulties shared by most NN-based PDE solvers, namely: (a) lack of efficient integration methods (especially in presence of singular solutions where Monte Carlo methods fail), and (b) lack of efficient optimizers,  which often fall in a local minima whose distance to global minima is uncertain. Other than that, numerical results look promising and are supported on a solid mathematical framework at the continuous level.

As future work,  \review{we will better analyze numerical integration techniques in the context of NNs; specifically, focused on Ritz-type minimizations.  In addition, we will investigate the behavior of optimal test functions to propose enhanced stopping criteria in the inner loop.  In particular, we believe that a global approximation of the trial-to-test operator with NNs would be helpful for this purpose.}

\section*{Acknowledgments}

This work has received funding from: the European Union's Horizon 2020 research and innovation program under the Marie Sklodowska-Curie grant agreement No 777778 (MATHROCKS); the Marie Sklodowska-Curie individual fellowship No 101017984 (GEODPG); the Spanish Ministry of Science and Innovation projects with references TED2021-132783B-I00, PID2019-108111RB-I00 (FEDER/AEI) and PDC2021-121093-I00 (AEI/Next Generation EU), the “BCAM Severo Ochoa” accreditation of excellence CEX2021-001142-S/MICIN/AEI/10.13039/501100011033; and the Basque Government through the BERC 2022-2025 program, the three Elkartek projects 3KIA (KK-2020/00049), EXPERTIA (KK-2021/00048), and SIGZE (KK-2021/00095), and the Consolidated Research Group MATHMODE (IT1456-22) given by the Department of Education.

\appendix
\section{Numerical \review{instability} of the min-max approach}\label{appendix}

Let $J:\U\setminus\{u^*\}\longrightarrow\V$ be the mapping that for each trial function returns the test maximizer of the actions of the residual $Bu-l\in\V'$ over the unitary sphere, i.e.,
\begin{equation}
J(u) := \arg\max_{\Vert v\Vert_\V = 1} (Bu-l)(v).
\end{equation} Notice that $J$ is not well defined at $u^*$ because $Bu^*-l$ is the null operator.

\begin{proposition}
$J(u)\in\V$ is unique and
\begin{equation} \label{appendix0}
J(u) = \frac{T(u-u^*)}{\Vert T(u-u^*)\Vert_\V}, \qquad u\in\U\setminus\{u^*\}.
\end{equation} 
\end{proposition}
\begin{proof} By the Riesz Representation Theorem,  for $Bu-l\in\V'$, there exists a unique $r_u\in\V$ such that
\begin{equation}\label{appendix1}
(r_u,v)_\V=(Bu-l)(v),\;\forall v\in\V\qquad \text{and} \qquad \Vert r_u\Vert_\V = \Vert Bu-l\Vert_{\V'} = \max_{\Vert v\Vert_\V = 1} (Bu-l)(v).
\end{equation} Let $w\in\V$ such that $\Vert w\Vert_{\V}=1$ and 
\begin{equation}
(Bu-l)(w)=\max_{\Vert v\Vert_{\V} = 1} (Bu-l)(v).
\end{equation} By the Cauchy-Schwarz inequality:
\begin{align}
0\leq \Vert Bu-l\Vert_{\V'} = (Bu-l)(w)=(r_u,w)_\V\leq \Vert r_u\Vert_{\V} \Vert w\Vert_{\V} = \Vert r_u\Vert_{\V},
\end{align} where the equality holds if and only if $w = \lambda r_u$ for some $\lambda\geq 0$.  Then,  $w=\frac{r_u}{\Vert r_u\Vert_\V}$ is the unique solution.  Employing equation \eqref{error_residual_representation}, we find \eqref{appendix0}. 
\end{proof}

\begin{proposition}
$J$ is not Lipschitz continuous around $u^*$, i.e.  there does not exist a constant $C>0$ such that 
\begin{equation}\label{appendix continuity}
\Vert J(u_1)-J(u_2)\Vert_\V \leq C\Vert u_1-u_2\Vert_\U, \qquad\forall u_1,u_2\in\U\setminus\{u^*\}.
\end{equation}
\end{proposition}
\begin{proof}
Assume by contradiction that there exists $C>0$ such that \eqref{appendix continuity} holds, and select $u_2=2u^*-u_1$ with $u_1\neq u^*$. Then,
\begin{equation}
2 = \left\Vert J(u_1) - J(u_2)\right\Vert_{\V} \leq C \Vert u_1-u_2\Vert_\U = 2 C \Vert u_1-u^*\Vert_\U.
\end{equation} Letting  $u_1\to u^*$, we obtain $C\to\infty$.
\end{proof}

\bibliography{references}

\end{document}